\documentclass[10pt,a4paper,reqno]{amsart}

\usepackage{tabls}
\usepackage{url}
\usepackage{multicol} 
\usepackage[linktocpage = true]{hyperref}
\hypersetup{
 colorlinks = true,
 citecolor = blue,
}
\usepackage{color}
\definecolor{red}{rgb}{1,0,0}
\definecolor{green}{rgb}{0,1,0}
\definecolor{blue}{rgb}{0,0,1}
\definecolor{refkey}{gray}{.625}
\definecolor{labelkey}{gray}{.625}
\usepackage[lite]{amsrefs}
\usepackage[a4paper]{geometry}
\usepackage{amsfonts}
\usepackage{amsmath,amssymb}
\usepackage{tablists}
\restorelistitem    %还原\item
\usepackage{amssymb}
\usepackage{mathrsfs}
\usepackage{amsmath}
\usepackage{amsthm}
\usepackage{amscd}
\usepackage{enumerate}
\usepackage{paralist}
\usepackage{graphicx}
\usepackage{mathtools}
\usepackage[all,cmtip]{xy}
\usepackage{kantlipsum}
\usepackage{tikz}
\usepackage{tikz-cd}
\allowdisplaybreaks

\def\F{\mathbb F}

\def\End{\mathrm{End}}
\def\Aut{\mathrm{Aut}}

\def\dim{\mathrm{dim}}
\def\tr{\mathrm{Tr}}
\def\id{\mathrm{Id}}
\def\Ker{\mathrm{Ker}}

\newcommand{\Ann}{\mathrm{Ann}}

\makeatletter

\theoremstyle{plain}
\newtheorem{thm}{\protect\theoremname}[section]
\theoremstyle{plain}

\theoremstyle{plain}
\newtheorem{cor}[thm]{\protect\corollaryname}
\theoremstyle{plain}
\newtheorem{lem}[thm]{\protect\lemmaname}
\theoremstyle{definition}

\theoremstyle{definition}
\newtheorem{defn}[thm]{\protect\definitionname}
\theoremstyle{plain}
\newtheorem{remark}[thm]{\protect\remarkname}
\theoremstyle{definition}
\newtheorem{notation}[thm]{\protect\notationname}
\AtBeginDocument{
	
}

\makeatother

  \providecommand{\corollaryname}{Corollary}
  \providecommand{\examplename}{Example}
  \providecommand{\lemmaname}{Lemma}
  \providecommand{\propositionname}{Proposition}
  \providecommand{\theoremname}{Theorem}
  \providecommand{\definitionname}{Definition}
  \providecommand{\remarkname}{Remark}
  \providecommand{\notationname}{Notation}

  \makeatletter
  \newcommand{\be}{%
  \begingroup
  \eqnarray%
   \@ifstar{\nonumber}{}%
  }

\makeatother

\title{A Modular Interpretation of BBGS Towers}
\author{Rui Chen}%{$^\diamond$}\\
\address{Yau Mathematical Sciences Center, Tsinghua University, Beijing, China}
\email{\href{chen-r15@mails.tsinghua.edu.cn}{chen-r15@mails.tsinghua.edu.cn}}

\author{Zhuo Chen}%$^*$\\
\address{Department of Mathematical Sciences, Tsinghua University, Beijing, China}%, Beijing 100871, China} }
\email{\href{chenzhuo@mail.tsinghua.edu.cn}{chenzhuo@mail.tsinghua.edu.cn}}
\thanks{Research partially supported by NSFC grant 11471179.}

\author{Chuangqiang Hu}
\address{Yau Mathematical Sciences Center, Tsinghua University, Beijing, China}
\email{\href{huchq@tsinghua.edu.cn}{huchq@tsinghua.edu.cn}}

\begin{document}
\begin{abstract}
	In 2000,  based on his procedure for constructing explicit towers of modular curves,  Elkies  deduced explicit equations of rank-$2$ Drinfeld modular curves which coincide  with the 
	asymptotically optimal towers of curves constructed by Garcia and Stichtenoth. In 2015, Bassa, Beelen, Garcia, and Stichtenoth constructed a celebrated	 (recursive and good) tower  (BBGS-tower for short) of curves  and outlined a modular interpretation of the defining equations. Soon after that,    Gekeler studied in depth  the  modular curves coming from sparse Drinfeld modules. 
	 In this paper, to establish a link  between these existing results, we propose and prove a generalized   Elkies' Theorem which tells in detail how to directly describe   a modular interpretation of the equations of rank-$m$ Drinfeld modular curves with $m\geqslant 2$. 
	%Moreover, we find the precise extension degree of the successive function fields of the BBGS-tower.

\end{abstract}

\maketitle{}


%\tableofcontents 
\section*{Introduction}

\subsection*{From Ihara's Quantity to Recursive Towers}
Estimation of the number of rational points on an algebraic curve over  the finite field $\F_q$ is an important subject  in number theory and algebraic geometry.  Let $C$ be a geometrically irreducible and smooth curve  over $\F_q$ and $g=g(C)$  its genus. The number of $\F_q$-rational points on $C$ has a well-known upper bound due to Hasse-Weil \cite{Weil1971}:  
% A. Weil – Courbes alg´ebriques et vari´et´es ab´eliennes, Herman, Paris, 1971.
\[\#(C(\F_q))\leqslant q+1+2g\sqrt{q}. \]
An improved bound is obtained by Serre \cite{Serre1984}:
\[\#(C(\F_q))\leqslant q+1+g[2\sqrt{q}].\]
 A curve   that attains the Hasse-Weil bound is called maximal. The interested reader is referred to \cite{H.Stichtenoth2009,Garcia2010,Giulietti2006,Giulietti2008,Cakcak2004,Cakcak2005} for standard examples of Hermitian, Garcia-G\"{u}neri-Stichtenoth, Giulietti-Korchm\'{a}ros, Suzuki,  and Ree curves, and to \cite{Skabelund2018, Beelen2018} for recent progress on maximal curves.
 
%Garcia, A., G¨uneri, C., Stichtenoth, H.: A generalization of the Giulietti-Korchm´aros maximal curve. Advances in Geometry 10(3), 427–434 (2010). 
%Giulietti, M., Korchm´aros, G.: A new family of maximal curves over a finite field. Mathematische Annalen 343, 229–245 (2008). 
   
%J.-P. Serre. Sur le nombre des points rationnels d’une courbe alg´ebrique sur un corps fini.C. R. Acad. Sci. Paris, 296:397–402, 1983,

%S. G. Vlˆadut and V. G. Drinfel’d. Number of points of an algebraic curve. Functional analysis and its applications, 17(1):53–54, 19830.

Ihara \cite{Ihara1982} noted that the Hasse-Weil bound    becomes weak when the genus $g$ is relatively large   with respect to the size $q$ of the base field $\F_q$. 
  Also in \cite{Ihara1982}, Ihara introduced an  asymptotic bound  of the number of rational points, now known as the Ihara's quantity:
\[A(q):=\limsup\limits_{ {g}\to \infty}\frac{N_q( {g})}{ {g}},\]
where \[N_q( {g})=\max \{\#(C(\F_q))| \text{$C$ is a curve over $\F_q$ with genus $ {g}$}\}.\]

 The upper bound of Ihara's quantity     
\[A(q)\leqslant \sqrt{q}-1 \]
was discovered 
by Drinfeld-Vl\u{a}du\c{t} \cite{Vladut1983}. Meanwhile, in search of lower bounds of $A(q)$, people have invented  varies  constructions of  towers of curves over $\F_q$.
Roughly speaking, a tower  $\mathcal{T}$ of curves over $\F_q$  consists of a  family of curves $C_n$ together with a sequence of successive surjective maps:%\footnote{One can also flip the diagram  to make  arrows point to the left.}:
 \[  
\begin{tikzcd}
C_1&C_2\arrow[l,"p_1"']&\arrow[l,"p_2"']\cdots&C_{n-1}\arrow[l,]&C_n\arrow[l,"p_{n-1}"']&
\cdots\arrow[l,"p_n"']
\end{tikzcd}
\]
such that all $C_n$ and $p_n$ are defined over $\F_q$ and   $g(C_n)\to \infty$  as $n\to \infty$.
The limit  
\[\lambda(\mathcal{T}):=\lim\limits_{n\to\infty}\frac{\#(C_n(\F_q))}{g(C_n)},\]
which always exists (see \cite[Lemma 7.2.3]{H.Stichtenoth2009}), certainly gives a lower bound of $A(q)$. 

A tower $\mathcal{T}$ is    called   (asymptotically)  \textit{good} if  $ \lambda(\mathcal{T}) > 0 $. 
Though it is  not normally easy to construct good towers, there are two approaches to construct them, either by class fields or by modular curves (classical, Shimura, and Drinfeld).    
In this paper, we will discuss three good towers arising from Drinfeld modules (elaborated in the subsequent Main Theorem).

Let us  list some remarkable lower bounds of  Ihara's quantity   achieved by good towers. 
\begin{enumerate}
\item Serre \cite{Serre1983} obtained the result 
\[A(q)\geqslant c\cdot \log q\] 
for some constant $c> 0$.  
A particular value $c=\frac{1}{96}$ appeared in  \cite{Niederreiter2001}.
%H. Niederreiter and C.-P. Xing, Rational points of curves over finite fields, Cambridge University Press, Cambridge (2001)..
\item For $q$ being small prime numbers, some known results are found by Angles and Maire \cite{Angles2002} ($A(5)\geqslant \frac{8}{11}$),  Hajir and Maire \cite{Hajir2000} ($A(3)\geqslant \frac{12}{25}$), Li and Maharaj \cite{Li2002} ($A(7)\geqslant \frac{9}{10}$, $A(11)\geqslant \frac{12}{11}$, $A(13)\geqslant \frac{4}{3}$, and $A(17)\geqslant \frac{8}{5}$), Niederreiter and Xing \cite{Niederreiter1998} ($A(2)\geqslant \frac{81}{317}$), Xing and Yeo \cite{Xing2007} ($A(2)\geqslant \frac{97}{376}$), and   Hall-Seelig  \cite{Hall2013} ($A(7)\geqslant \frac{12}{13}$ and $A(11)\geqslant \frac{8}{7}$). This list is not complete. 
\item   For square numbers $q $, a sharp bound is discovered:  $A(q)\geqslant \sqrt{q}-1$ (hence $A(q)=\sqrt{q}-1$),   independently, by Ihara  \cite{Ihara1982} 
and Tsfasman, Vl\u{a}du\c{t}, and  Zink  \cite{Tsfasman1982},   one  
using   families of Shimura modular curves, the other using   
families of classical modular curves.
By
Gekeler \cite{Gekeler2004},  certain families of Drinfeld modular curves  also attain this lower bound. 
\item When $q$ is a cubic number, say $q=p^3$,    Zink \cite{Zink1985} got the result   $A(q)\geqslant \frac{2(p^2-1)}{p+2}$ under the assumption that  $p$ is a prime.   Bezerra, Garcia, and Stichtenoth  \cite{Bezerra2005} proved  that this inequality holds for arbitrary cubic numbers $q$.
\item When $q=p^{2m+1}$ where $m\geqslant 1$,  Bassa, Beelen, Garcia, and Stichtenoth \cite{Bassa2015}  proved  that 
\begin{equation}\label{Eq:lbqm}
A(q)\geqslant \frac{2(p^{m+1}-1)}{p+1+(p-1)/{(p^m-1)}},
\end{equation}  which is a   source of inspiration of the present paper.
\end{enumerate}By   Goppa's construction \cite{Goppa1981},
good towers yield 
good linear error-correcting codes.    A celebrated discovery by Tsfasman \textit{et al.} \cite{Tsfasman1982} --- the existence of long linear codes with the relative parameters above the well-known Gilbert-Varshamov bound  \cite[Proposition 8.4.4]{H.Stichtenoth2009}, provided a vital link  between Ihara's quantity   and the realm of coding theory.

Good towers that are recursive  play   important roles in the  studies of  Ihara's quantity,  coding theory, and cryptography \cite{Hu2016,Hu2017,Hu2019,Cascuso2014,Conny1997,Aleshnikov1999}.     A tower $\mathcal{T}$ is called \textit{recursive} by an absolutely irreducible polynomial $f(x,y)\in\F_q(x)[y]$ (see \cite[Sections 3.6 and 7.2]{H.Stichtenoth2009}) if 
\begin{enumerate}
	\item  The initial curve $C_1$ is the projective line with coordinate $x_1$;
	
	\item For  $n\geq 2$,   
	 $C_n$ is the nonsingular projective model of an affine   curve defined by 
	 	 %normalization of the curve in $C_1^n$ consisting of $n$-tuples $(x_1, \ldots,x_n )$  such that  
	 \[f(x_1,x_2)=f(x_2,x_3)=\cdots=f(x_{n-1},x_n)=0. \] 
	   \end{enumerate}

A first concrete example of good tower which is recursive over $\F_{q^2}$  is given in 1995  by
 Garcia and Stichtenoth \cite{Garcia1995}
%A. Garcia and H. Stichtenoth, A tower of Artin-Schreier extensions of function fields attaining the Drinfeld-Vladut bound,Invent. Math.121(1995), 211-222.
with the  recursive polynomial    
\begin{align}\label{Eq:GST95}
f(x,y)= x^{q-1}y^q+ y-x^{q}.
\end{align}
Soon after that, they gave another tower  with the recursive function \cite{GarciaStichtenoth1996}  
\begin{align}\label{Eq:GST96}
f(x,y)=y^q+y-\frac{x^q}{x^{q-1}+1}, 
\end{align}
which turns out to be a subtower of the previous one. An excellent fact is that each of the two towers fulfils    $\lambda(\mathcal{T})=q-1$, the lower bound. We call such kind of towers   \emph{optimal}. 

The subject for investigation in this paper is concerned with   what 
Bassa, Beelen, Garcia, and Stichtenoth had 
presented in \cite{Bassa2015,A.Bassa2014} --- a general construction of recursive towers over non-prime fields. For short, we call them  \textit{BBGS towers}. Below is a brief account of such towers.

Suppose that $m=j+k\geqslant 2$ is a positive integer,  where $ j$ and $ k $ are coprime positive integers. Let $a$ and $b$ be   non-negative integers such that $ak-bj=1$.  Consider the  tower  $\mathcal{F}$ (over $\F_{q^m}$), respectively  $\mathcal{H}$,  arising from the recursive polynomial 
\begin{equation}\label{Eq:TBBGSF}
 {\mathcal{F}}(x,y)=\tr_j\left(\frac{y}{x^{q^k}}\right)+\tr_k\left(\frac{y^{q^j}}{x}\right)-1,\end{equation}
respectively
\begin{equation}\label{Eq:TBBGSH}
{\mathcal{H}}(x,y)=\frac{\tr_j(y)-a}{\tr_j(x)^{q^k}-a}-\frac{\tr_k(y)^{q^j}-b}{\tr_k(x)-b},
 \end{equation}
where $\tr_l(x):=\sum_{i=0}^{l-1}x^{q^i}$.  
%{ and } \[N_n:=\frac{q^n-1}{q-1}. \]	
A key result in \cite{A.Bassa2014} is  the inequality
\begin{equation*}
%\label{Eq:lambdalambda2}
\lambda(\mathcal{H})\geqslant\lambda(\mathcal{F})\geqslant 2\left(\frac{1}{q^k-1}+\frac{1}{q^{j}-1}\right)^{-1}
\end{equation*} 
from which one obtains  the lower bound in Equation  \eqref{Eq:lbqm}. Note that  the $m=2$ case of the tower $\mathcal{F}$ coincides with the one constructed by Equation \eqref{Eq:GST95}.

\subsection*{The Motivation and Main Result}
Part of the motivation behind this work is to better understand 
the modular interpretation of the BBGS towers presented in \cite{Bassa2015}. We are also inspired by
the other two works --- one is \cite{Elkies2001} by Elkies, where it is shown that the towers in \eqref{Eq:GST95}
 and \eqref{Eq:GST96} both arise from Drinfeld modular curves, the other is a recent work \cite{Nurdagul2017} by Anbar, Bassa, and Beelen, where   a particular tower $\mathcal{H}$ in \eqref{Eq:TBBGSH} with $(m,j,k)=(3,2,1)$ is investigated and proved to be modular. 
 
 It is natural to ask whether one can work out explicit modular explanations  of the BBGS towers $\mathcal{F}$ and $\mathcal{H}$ in, respectively,  \eqref{Eq:TBBGSF} and \eqref{Eq:TBBGSH} with general $(m,j,k)$, and if so, what information can be derived from such explanations.  
For this purpose, the present paper will follow a framework described by Gekeler  who proposed an abstract construction of Drinfeld modular curves \cite{Gekeler2019}, and our answer is  an explicit description of the relevant curves.

\vskip0.5cm
{\textbf{The Main Theorem}} [Generalized Elkies' Theorem]\label{Thm:A}
{\textit{Let $\mathcal{F}$  and   $\mathcal{H}$ be functions defined as earlier  by Equations \eqref{Eq:TBBGSF} and \eqref{Eq:TBBGSH}. 	Let ${X}_{m,j}(T^n)$, $\dot{X}_{m,j}(T^n)$, and $\ddot{X}_{m,j}(T^n)$ be Drinfeld modular curves  defined as in Definition  \ref{Defn:normalizedElkiesModularCurve}. Assume that $k$ is not divided by  the characteristic $p$.   Then,
	\begin{enumerate}[A.]
		\item The function field $\ddot F^{(n)}_{m,j}$ of the Drinfeld modular curve $ \ddot{X}_{m,j}(T^n) $ over $\F_{q^m}$ is generated by  variables $x_1, x_2, \ldots , x_{n} $ that are subject to the following recursive equations 
		 \begin{equation}\label{Eq:A,m,j}
		\mathcal{F}(x_{i-1},x_{i})=0,\qquad  i=2,3,\ldots,n.
		\end{equation} 
		\item The function field $\dot F^{(n)}_{m,j}$ of the Drinfeld modular curve $ \dot{X}_{m,j}(T^n) $ over $\F_{q^m}$ is generated by   variables $X_1, X_2, \ldots , X_{n} $ that are subject to the following recursive equations 
		 \begin{equation}\label{Eq:dA,m,j}
		\mathcal{G}(X_{i-1},X_i)=0,\qquad i=2,3,\ldots,n,
		\end{equation}
		where
				\[{\mathcal{G}}(x,y)=y\left(\sum_{i=0}^{j-1}\frac{y^{N_i}}{x^{N_{k+i}}}+\sum_{i=j}^{m-1}\frac{y^{N_i}}{x^{N_{i-j}}}\right)^{q-1}-x ~\mbox{ and }~ N_l=\frac{q^l-1}{q-1} .\]
						\item The function field $F^{(1)}_{m,j}$ of $X_{m,j}(T)$ over $\F_{q^m}$ equals the rational function field $\F_{q^m}(z)$ with variable $z$. 
		If $n\geqslant 2$, then the function field $F^{(n)}_{m,j}$ of $X_{m,j}(T^n)$ over $\F_{q^m}$ is generated by variables $u_2,\ldots,u_{n}$ satisfying the recursion
	\begin{equation}\label{Eq:HU1U2}
		\mathcal{H}(u_{i-1},u_i)=0,\qquad i=3,4,\ldots,n.
	\end{equation}
			\end{enumerate}
}}

We  remark that\begin{enumerate} \item Parts A and B of the Main Theorem  can be  adapted to arguments over the base field $\F_{q}$;
\item
The original Elkies' Theorem in \cite[Section 4]{Elkies2001}  corresponds to the  $(m,j,k)=(2,1,1)$ case; 

\item It is tempting to mimic  Elkies' approach to handle the general $(m,j,k)$ cases ($m\geqslant 2$). However, it does not simply yield what the theorem desired. Instead, we find another but  equivalent description of Drinfeld modular curves, and thereby obtaining recursive formulas of the corresponding curves. %CHECK THESE WORDS!
\end{enumerate}

From the Main Theorem described as above, we are able to find supersingular points (which are necessarily rational) on the three modular curves  (see Remark \ref{Rmk:last}). We hope our  result  will shed   light on finding explicit equations of certain  towers, and in particular, provide valuable geometric insight into the nature of  the BBGS-towers. 

To this day, little is known about the general description of recursive towers from modular curves. Li, Maharaj,  Stichtenoth, and Elkies   \cite{H.Stichtenoth2002} exhibited four  optimal towers over $\F_{p^2}$ ($p=2,3,5,7$); Garcia, Stichtenoth, and  R\"{u}ck \cite{Garcia2003} computed an optimal tower over $\F_{p^2}$; Hasegawa, Inuzuka, and Suzuki \cite{Hasegawa2012,Hasegawa2013,Hasegawa2017} provided a number of classical and Shimura modular curves by using Elkies' procedure;
%In Perter... \cite{}, a particular case so that the inequality in \eqref{Eq:lambdalambda2} becomes an equality, is presented.  
Hallouin and Perret \cite{Hallouin2016} proposed a systematic method to produce potentially good recursive towers over finite fields. Our result and approach should be useful in the studies of  Drinfeld modular curves in a wider range.    
In fact, based on the current work,    we have come up with a sequence of Drinfeld modular curves which are organized in an elegant manner.

We also would like to point out works of others that are related to the present paper.
In the work of Hu and Zhao \cite{Hu2017,Hu2016}, varies bases of certain Riemann-Roch spaces associated to  the BBGS tower  $\mathcal{F}$ are investigated.  The interlink between explicit towers and modular curves   emerges  in Elkies' works \cite{Elkies1997,Elkies2001,Elkies2002}, leading to the \textit{Elkies' modularity  conjecture}  --- All asymptotically optimal recursive   towers defined over $\F_{q^2}$   arise from reductions of elliptic,
Shimura, or Drinfeld modular curves.

This paper is organized as follows.  
Section \ref{Sec:basic} gives a succinct account of standard facts
about Drinfeld modules, whose purpose is
to fix the notation. 
The 
Drinfeld modular curves $X_{m,j}(T^n)$, $\dot{X}_{m,j}(T^n)$, $\ddot{X}_{m,j}(T^n)$, and $\ddot{M}_{m,j}(T^n)$ ($n\geqslant 1$) are defined in Section \ref{Sec:DMC}  and 
a relation between $\ddot{X}_{m,j}(T^n)$ and $\ddot{M}_{m,j}(T^n)$ is then proved.  Sections  \ref{Sec:basic} and \ref{Sec:DMC} also establish a list of important facts and identities that are
 subsequently used in Section \ref{Sec:Main part}  to prove
the statements of our Main Theorem.

\subsection*{Acknowledgements}

Hu would  like to thank Binglong Chen, Chang-An Zhao, and Chaoping Xing
for useful discussions and comments. We gratefully  acknowledge the  critical  review by the anonymous reviewer   on the first version of the manuscript.

\section{Preliminaries}\label{Sec:basic}
\subsection{Drinfeld Modules}
In this part we give a brief introduction to the notion of Drinfeld module, which was  introduced by Drinfeld in his celebrated work  \cite{Drinfeld1974}. For more in-depth studies, please see \cite{Drinfeld1974,Gekeler1986,Goss1996,Villa2006,Dale2018}.
%V.G. Drinfeld, Elliptic modules, Math. USSR Sb. 23 (1974) 561–592.
    
%D. Goss, \Basic structures of function eld arithmetic", Springer Berlin Heidelberg, 1996.
  
Some notations are in order. 
Let $\F_q$ be a finite field of cardinality $q$. Denote by $A:=\F_q[T]$ the polynomial ring over $\F_q$. Let  $ L $ be a field containing $ \F_q $ together with a fixed $ \mathbb{F}_{q} $-algebra homomorphism $ \iota : A \to L $.  Denote by $ L\{ \tau \} $ the non-commutative $ L $-algebra which is generated by the $q$-Frobenius endomorphism $ \tau $ such that  $ \tau \cdot a = a^{q } \tau  $ for all $ a \in L$. We refer to the $L$-algebra $ L\{ \tau \} $   as a \textit{twisted polynomial ring} (also known as an Ore ring \cite{Ore1933}). 
Denote by $ \mathbb{G}_a $ the additive group scheme over $ L $. It is standard that the ring of $\F_q$-linear endomorphisms  $ \End_{\F_q} (\mathbb{ G}_a )  $ of $ \mathbb{G}_a $ is isomorphic to $ L\{ \tau \} $. 

Let $ \bar L $  be an  algebraic closure of $ L $. By restricting $ \End_{\F_q} (\mathbb{ G}_a ) $ to the $ \bar{L} $-geometric points of $ \mathbb{ G}_a $, we obtain an induced   action of $ L \{\tau \} $ on $\bar{L} $. Explicitly, the action of a twisted polynomial $ f = \sum_{i=0}^{m} g_i \tau^i \in L\{\tau\}  $  on $ \bar{L} $ is given by
\[
f: \bar{L} \to \bar{L} , \qquad \mu \mapsto f(\mu):=\sum_{i=0}^{m} g_i \mu^{q^i}  .
\]
 The kernel of   $f$  is defined and denoted by	$$\Ker (f):=\{\mu\in \bar L | f(\mu)=0\}, $$ which is a finite dimensional  $\F_q$-linear subspace of $\bar L$. 
 %\begin{remark}
%\begin{remark}
	It is more suitable to consider $ \Ker(f)$ 
	as a group subscheme of $ \mathbb{G}_a $, rather than   a subgroup of $ \bar L $. However,  we shall not need this refinement in this work.
%\end{remark}
% 	Indeed, as a subgroup of $ \bar L $, it is more suitable to consider $ \Ker(f)$ as a group subscheme of $ \mathbb{G}_a $.   %?????  DON'T UNDERSTAND THIS %Nevertheless, we don't  need this refinement.
 %\end{remark}
 
 The \textit{point derivation} $\partial_0$ of a twisted polynomial $f$ at $ 0 $ is  standard:
 \[
 \partial_0:\quad L\{\tau \} \to L , \qquad f=\sum_{i=0}^{m} g_i \tau^i \mapsto g_0.
 \]
 Note  that $\partial_0$ is a homomorphism of $ \mathbb{F}_{q} $-algebras.

 A $\textit{Drinfeld module}$ over $L$  is  an $ \mathbb{F}_{q } $-algebra homomorphism
  \[\phi :  A \to L\{ \tau \},\qquad a\mapsto \phi_a\,,\]
   satisfying the conditions
   \begin{enumerate}
     \item  there exists $ a \in A $ such that $ \phi_ a \not = \iota (a )$; and
     \item $\partial_0 \circ \phi = \iota $, i.e., the following diagram of $ \mathbb{F}_{q } $-algebra homomorphisms
     \begin{equation*}
      \begin{tikzcd}
     A\arrow[rd,"\iota"] \arrow[r,"\phi"] & L\{\tau\}\arrow[d,"\partial_0"]\\
     & L
   \end{tikzcd}
   \end{equation*}
   is commutative.
   \end{enumerate}

For a Drinfeld module $\phi$ as above, the kernel of $\iota$, which is a prime ideal in $A$, is called   the \textit{characteristic} of $\phi$.   
As $ A$ is the polynomial ring $\F_q[T] $,   a Drinfeld module $ \phi $ is uniquely determined by a twisted polynomial $\phi_T$ over $L$. We suppose that
  \begin{equation*}%\label{Eq:phiT}
 \phi_T =   g_m \tau^m +\cdots + g_2 \tau^2 + g_1 \tau + g_0\,,
     \end{equation*}
    where $ g_ m$  is not $0 $ for some integer $m>0$. The number $ m  $ is called the \textit{rank} of   $ \phi $. If %$g_0\neq 0$ and 
    $ g_i = 0 $ for $ 1 \leqslant i \leqslant m-1 $, then $\phi$ is said to be  $\textit{supersingular}$.

   For a polynomial  $a\in A$, the kernel of $\phi_a $ is an  $ A $-submodule of $\bar{L}$ thanks to the commutativity    $\phi_a\phi_b=\phi_b\phi_a$, for all $ b \in A $.    In a special situation described below, the $A$-module structure of $\Ker (\phi_a)$ is explicit.

 \begin{lem}[{\cite[Proposition I.1.6]{Gekeler1986}}]\label{Lem:moduleiso}
 %D. Festi Notes on the arithmetic of Drinfeld modules  2013
   If  $a\in A$ is coprime to the characteristic of $\phi$, then
 \[ \Ker ( \phi_{a} ) \cong  (A/(aA))^{\oplus m} \]
 as $A$-modules.
\end{lem}

 One recognizes this fact in parallel  with a well-known result  of elliptic curves:
 $$
 E[n]\cong \mathbb{Z}/(n\mathbb{Z}) \oplus \mathbb{Z}/(n\mathbb{Z}),
 $$
 where $E[n]$ is the group of $n$-torsion points on an elliptic curve $E$.

 %See \cite{???} for more details.

%The most interesting example of isogenies is of the form $\lambda:= a- \tau^k  $ for some $a\in L$. For further details, we propose the following lemma which is originally given by  Bassa \textit{et al.}.
%Towers of function elds over non-prime nite elds, Mosc. Math. J. 15 (2015), no. 1, 1{29,181. MR3427409
%\begin{lem}[\cite{Bassa2015}, Proposition $4.1$]
%	Let $ \phi$ and $\psi $ be two Drinfeld modules which are defined by
%	\[\phi_T = - \tau^m + g_j \tau^j + 1,\]  and
%	\[ \phi'_T = - \tau^m + g'_j \tau^j + 1. \]
%Then 	$  \lambda:= a- \tau^k   \in L\{ \tau \} $ is an isogeny from $ \phi $ to $ \psi $  if and only if
%	\[ g'_j = -a^{q^m } + a + g_j^{q^k } \quad \text{and} \quad  a g_j =g'_j a^{q^j} .\]
%\end{lem}
%\begin{proof}
%	From the definition, $\lambda\phi=\psi\lambda$ implies that
%	\begin{equation}
%	\begin{aligned}
%	&\left(a-\tau^k\right)\left(-\tau^m+g_j\tau^j+1\right)\\
%	=&\tau^{m+k}-\left({g_j}^{q^k}+a\right)\tau^{m}-\tau^k+a{g_j}\tau^j+1\\
%	=&\tau^{m+k}-\left(g'_j +a^{q^m}\right)\tau^{m}-\tau^k+a^{q^j}g'_j \tau^j+1\\
%	=&\left(-\tau^m+g'_j \tau^j+1\right)\left(a-\tau^k\right).
%	\end{aligned}
%	\end{equation}
%	Comparing the coefficients, we get
%	\[ g_j'= -a^{q^m } + a + g_j^{q^k } \quad \text{and} \quad  a g_j = g_j' a^{q^j} .\]
%\end{proof}

  \subsection{Isomorphisms of Drinfeld Modules}
  We make some important conventions in subsequent analysis:
  \begin{itemize}
  	\item We assume that $\F_{q^m}\subseteq L$;
  	\item We only consider 
  	Drinfeld modules of the form
  	\begin{equation}\label{Eq:type}
  	\phi_{T} = g_m \tau^m + g_j \tau ^j + 1.
  	\end{equation}
  	Here  $m=j+k\geqslant 2$, and $ j$ and $ k$ are mutually coprime positive integers.  
  	  \end{itemize} 
    By the second assumption, the characteristic  of  a Drinfeld module   is the ideal $  ( T-1 ) $. In other words,   $ \iota $ maps $T \in A$ to $ 1 \in L$. 
 Indeed, Bassa \emph{et al.} studied this type of Drinfeld modules 
 in \cites{A.Bassa2014,Bassa2015}.
  Clearly,   $\phi$ is of rank    $m$.  We call $\phi$ \textit{normalized} if $g_m=-1$. 
  \begin{notation}
  We denote by $  \mathcal{D}_{m,j} $ the set of  normalized Drinfeld modules that are of the form \begin{equation}\label{Eq:normalizedtype}
  \phi_{T} = - \tau^m + g_j \tau ^j + 1.
  \end{equation}
  	
  \end{notation}
   
 \begin{defn}
 Two Drinfeld modules $\phi$ and $\psi$   over $L$  are said to be \textit{isomorphic} over $\bar L$, if   there exists an element $\lambda\in \bar L^*$ such that for all $a\in A$, the equation
 \begin{equation}\label{Eq:defnisomorphic}
 \lambda\phi_a=\psi_a\lambda
 \end{equation} 
 holds in $\bar L\{\tau\}$.
 \end{defn}
 
Certainly, Equation (\ref{Eq:defnisomorphic}) amounts to the condition  $\lambda\phi_T=\psi_T\lambda$.	
  A  Drinfeld module  of the form (\ref{Eq:type}) 
    is  isomorphic to   a  normalized one over $\bar L$ provided that $(-g_m) $ is a $(q^m-1)$-st power.
  In fact, one takes $\lambda\in   L$ which is a  root of $(-g_m)$ of order $(q^m-1)$.   Then the scalar multiplication by $\lambda$ gives   an isomorphism from $\phi$ to the normalized Drinfeld module $\psi$   by setting
  \[\psi_T=-\tau^m+\lambda^{1-q^j}g_j\tau^j+1, \]
  because $\lambda\phi_T=\psi_T\lambda$.

  Recall that $N_m=\frac{q^m-1}{q-1}$. We call $ J(\phi): = g_j^{N_m} \in L$  the $ J $-\textit{invariant} of the normalized Drinfeld module  $ \phi \in  \mathcal{D}_{m,j}  $.  
 A well-known fact is that the isomorphism class of an elliptic curve is completely determined by its $j$-invariant. A similar fact  for Drinfeld modules is the following   
 \begin{lem}[{\cite[Section 4]{Bassa2015}}]\label{Lem:Isomorphic}
 	  	 For two normalized Drinfeld modules   $ \phi$ and  $ \phi^{\prime} $ represented, respectively, by
 		\begin{equation}
 		\label{Eq:phiTandphiT'} \phi_{T} = -\tau^m + g_j\tau ^j +1  \mbox{~and~}~ \phi_{T}^{\prime} = -\tau^m + g_j^{\prime}\tau ^j +1,
 		\end{equation}
 		the following statements are equivalent:
 		\begin{enumerate}
 			\item The Drinfeld modules $\phi$ and $\phi'$ are isomorphic over $\bar L $;
 			\item There exists some $ \lambda \in  \mathbb{F}_{q^m}^{* } $ such that $ g_j =g_j' \lambda^ {q^j -1 } $;
 			\item The $J$-invariants of $\phi$ and $\phi'$ coincide: $ J(\phi) = J(\phi^{\prime}) $.
 		\end{enumerate} 
   \end{lem}
We sketch a proof for completeness.
 \begin{proof}The implication (1) $\Leftrightarrow$ (2) is easy.
 	 	  If $\lambda \phi_T=\phi'_T\lambda$ holds for some $\lambda\in \bar L$,  then
 	\[-\lambda\tau^m+g_j\lambda\tau^j+\lambda=-\lambda^{q^m}\tau^m+g_j'\lambda^{q^j}\tau^j+\lambda.\]
 	Hence, $\lambda^{q^m-1}=1$ and $g_j=g_j' \lambda^ {q^j -1 } $. The converse   is also obvious.
 	
 It is   straightforward to see the implication (2) $\Rightarrow$ (3).	If  $g_j =g_j' \lambda^ {q^j -1 } $ holds for some $ \lambda \in  \mathbb{F}_{q^m}^{* } $, then \[{g_j}^{N_m}=\left({g_j' \lambda^ {q^j -1 }}\right)^{N_m}=({g_j' })^{N_m}.\]

 	We finally show the implication (3) $\Rightarrow$ (2).
 	Note that if $g_j=g'_j=0$ (i.e. supersingular Drinfeld modules), the proof is trivial. Below we assume that  $g_j'\neq 0$.

 	If $ J(\phi) = J(\phi^{\prime}) $ holds, then  
 	\[\left(\frac{g_j}{g'_j}\right)^{N_m}=1, \] and  hence $g_j/{g'_j}\in (\mathbb{F}_{q^m}^* ) ^{q-1}$.
 	Since $  m $ and $j $ are coprime, the image of { the map}
 	\[ \F_{q^m}\to\F_{q^m},\qquad \mu \mapsto \mu^{q^j-1} \]
 	is identically $ (\mathbb{F}_{q^m} ) ^{q-1}$. Therefore, there exists  $\lambda\in\F_{q^m}^*$  such that
 	$\frac{g_j}{g'_j}=\lambda^{q^j-1}$, as required. 
 \end{proof}

 \subsection{Isogenies of Drinfeld Modules}
 \begin{defn}
    An \textit{isogeny} of Drinfeld modules from $ \phi $ to $ \psi $ is a twisted polynomial $ \lambda \in \bar{L} \{ \tau \}  $ such that for all $a \in A$, the equation
 \begin{equation}\label{Eq:defnisogeny}
 \lambda\phi_a = \psi_a\lambda    
 \end{equation}
 holds in $\bar{L} \{ \tau \}$.
\end{defn}
 	
 	Apparently, Equation \eqref{Eq:defnisogeny} amounts to the condition   $\lambda\phi_T=\psi_T\lambda$. 
 	In this case,    $ \Ker(\lambda) \subseteq \bar{L}$    admits an $ A $-module structure which is defined by
	\[a\cdot \mu:=\phi_a(\mu),\qquad \text{ for $a\in A,~\mu\in\Ker(\lambda)$}.\]
	Here     $\phi_a(\mu)$ on the right hand side belongs to $\Ker(\lambda)$,  by Equation \eqref{Eq:defnisogeny}.
	
	%SO THE A-MODULE STRUCTURE OF KER LAMBDA DEPENDS ON PHI??? WHY NOT PSY???

\begin{notation}  Let us  set up some constantly used notations.    For $0\neq x\in \bar L$, define	three types of twisted polynomials over $\bar L$:
	\begin{enumerate}
		\item $\eta_x: =1+x^{1-q}\tau+x^{1-q^2}\tau^2+\cdots+x^{1-q^{k-1}}\tau^{k-1}$;
				\item $\lambda_x:=x^{q^k-1}-\tau^k=(x^{q^k-1}-x^{q^k-q}\tau)\eta_x$; and
			\item \begin{flushleft} $Q_x:=x^{1-q^k}+x^{1-q^{k+1}}\tau+\cdots+x^{1-q^{m-1}}\tau^{j-1}  +\tau^j+x^{1-q}\tau^{j+1}+\cdots+x^{1-q^{k-1}}\tau^{m-1}$.\end{flushleft}
			  					 			\end{enumerate}  
\end{notation}
Now we can reformulate the function $\mathcal{F}$ defined by  Equation \eqref{Eq:TBBGSF}:$$\mathcal{F}(x,y)=\frac{Q_{x}(y)}{x}-1.$$
 \begin{notation}\label{Notation:phix}
	For $0\neq x\in \bar L$, let $\phi^x$ be the Drinfeld module in $\mathcal{D}_{m,j}$ such that $\phi^x_T(x)=0$. In other words,   $\phi^x$ is 
	represented by
	\[\phi_T^x:=-\tau^m+g(x)\tau^j+1,\] 
	where $g(x)=x^{q^m-q^j}-x^{1-q^j}$.
\end{notation}  
We need a lemma which is generalized from \cite[Equation (11)]{Elkies2001} (for the $(m,k)=(2,1)$ case) and \cite[Section 3.1]{Nurdagul2017} (for the $(m,k)=(3,1)$ case).
\begin{lem}\label{Lem:etaphiQlam}
	Let $\eta_x$, $Q_x$, $\phi_T^x$,  and   $\lambda_x$ be defined as earlier. We have
	\[   \eta_x\phi^x_T=Q_x \lambda_x.\]
	%In particular, for $k=1$, the equality becomes $\phi_T=Q_x \lambda_x$.
\end{lem}
\begin{proof}
	 The proof is by direct calculations. Let us first assume that $k> j$. On the one hand,  we have
	\begin{equation*}
	\begin{aligned}
	\eta_x\phi_T^x=&\left(1+x^{1-q}\tau+x^{1-q^2}\tau^2+\cdots+x^{1-q^{k-1}}\tau^{k-1}\right)\left(-\tau^m
	+g(x)\tau^j+1  \right)\\
	=&\sum_{s=0}^{j-1}x^{1-q^s}\tau^s+\sum_{s=j}^{k-1}(x^{1-q^s}+x^{1-q^{s-j}}g(x)^{q^{s-j}})\tau^s\\
	&+\sum_{s=k}^{m-1}\left(x^{1-q^{s-j}}g(x)^{q^{s-j}}\right)\tau^s-\sum_{s=m}^{2m-j-1}x^{1-q^{s-m}}\tau^s.
	\end{aligned}
	\end{equation*}
	On the other hand, we have
	\begin{equation*}
	\begin{aligned}
	Q_x \lambda_x=&\left(x^{1-q^k}+x^{1-q^{k+1}}\tau+\cdots+x^{1-q^{k+j-1}}\tau^{j-1}
	\notag\right.
	\\
	\phantom{=\;\;}
	&\left.+\tau^j+x^{1-q}\tau^{j+1}+\cdots+x^{1-q^{k-1}}\tau^{m-1}\right)\left(x^{q^k-1}-\tau^k\right)\\
	=&\sum_{s=0}^{j-1}x^{1-q^s}\tau^s+\sum_{s=j}^{k-1}x^{1-q^{s-j}}x^{q^{k+s}-q^s}\tau^s\\
	&+\sum_{s=k}^{m-1}\left(-x^{1-q^s}+x^{1-q^{s-j}}\cdot x^{q^{k+s}-q^{s}}\right)\tau^s-\sum_{s=m}^{2m-j-1}x^{1-q^{s-m}}\tau^s.
	\end{aligned}
	\end{equation*}
	By carefully examining coefficients of the relevant $\tau^s$-terms, we see that $\eta_x\phi^x_T$ and $Q_x \lambda_x$ are identical.
	
	If  $k\leqslant j$, the  statement   is proved in a similar manner. 
	\end{proof}
The following theorem is a minor modification of a result obtained in \cite[Proposition 4.2]{Bassa2015}.  
  \begin{thm}\label{Thm:Isolamx}Let $x$, $y\in\bar L$ be nonzero elements   satisfying  $Q_x(y)=x$. Then the twisted polynomial $\lambda_x$ is an isogeny from $\phi^{x}$ to $\phi^{y}$.
 \end{thm}
 \begin{proof}
We wish to show that $(\lambda_x\phi^x_T-\phi^y_T\lambda_x)$ is identically zero. Note that
 	 	\begin{align*}
 	\lambda_x\phi_T^x&=\left(x^{q^k-1}-\tau^k\right)\left(-\tau^m+g(x)\tau^j+1\right)\\
 	&=\tau^{m+k}-\left({g(x)}^{q^k}+x^{q^k-1}\right)\tau^{m}-\tau^k+{x}^{q^{k-1}}{g(x)}\tau^j+1,
  \end{align*}
and
\begin{align*}
 	 	\phi_T^{y}\lambda_{x}&=\left(-\tau^m+g(y)\tau^j+1\right)\left({x}^{q^k-1}-\tau^k\right)\\
 	&=\tau^{m+k}-\left(g(y)+x^{q^{k+m}-q^m}\right)\tau^{m}-\tau^k+{x}^{q^{m}-q^j}g\left(y\right)\tau^j+1.
 	\end{align*}
  Hence we have
  	\begin{equation}\label{Eq:lamxphixphiy}
 \lambda_x\phi_T^x-\phi_T^{y}\lambda_{x}
 =g_m\tau^m-g_j\tau^j,
  \end{equation}
  where 
  \[g_m:=-\left({g(x)}^{q^k}+x^{q^k-1}\right)+\left(g(y)+{{x}^{q^{k+m}-q^m}}\right),\]
  and
  \[g_j:={x}^{q^{k-1}}{g(x)}-{x}^{q^{m}-q^j}g\left(y\right).\]
 
 We will show that $g_m$ and $g_j\in \bar L$ are both identically zero.  For
this purpose, let us set up a polynomial $G(X):=g_mX^{q^k-1}-g_j \in \bar L[X]$ and it suffices to show that $G(X)=0$.

  Let $H_{x,y}:=\{h^{q^j} |h\in \bar L, \lambda_x(h)=y\}$ be a subset in $\bar L$. Obviously, it has exactly $q^k$ elements.
  
   Take an element $h^{q^j}\in H_{x,y}$.  We observe that
   \begin{align*}
   	h^{q^j}G(h^{q^j})=& \left(\lambda_{x}\phi_T^{x}-\phi^y_T\lambda_x\right)(h)\quad \mbox{(by Equation \eqref{Eq:lamxphixphiy})}\\
   	=&\lambda_{x}\phi_T^{x}(h)
   	-\phi^y_T\lambda_x(h) \\=&
   	(x^{q^k-1}-x^{q^k-q}\tau)\eta_{x}\phi_T^{x}(h) -\phi^y_T(y)
   	\\
   	=&(x^{q^k-1}-x^{q^k-q}\tau)Q_{x}\lambda_{x}(h)\quad\mbox{(by Lemma \ref{Lem:etaphiQlam})}\\
   	=&(x^{q^k-1}-x^{q^k-q}\tau)Q_{x}(y)\\
   	=&(x^{q^k-1}-x^{q^k-q}\tau)(x)
   	=x^{q^k-1} x-x^{q^k-q} x^q=0 .
   	\end{align*}
  
  It follows that $G(h^{q^j})=0$ and thus $G$ has at least   $q^k$ zeros. However, by construction,   $G$ is of   degree   $(q^k-1)$. Hence $G$ must be trivial.  
 \end{proof}
 \begin{cor}\label{Cor:isolmdphiPhi}
 	 Let  $x_1,\ldots,x_n$ be  nonzero elements in $\bar L$ such that  $Q_{x_i}(x_{i+1})=x_i$  for $i=1,\ldots,n-1$.
 	 Then the twisted polynomial $\lambda_{x_{n-1}}\cdots\lambda_{x_1}$ is an isogeny from $\phi^{x_1}$ to $\phi^{x_n}$. 
 	\end{cor}
 \begin{proof}
 	We  use  Theorem \ref{Thm:Isolamx} repeatedly:
 	%\begin{equation}
 	\begin{align*}
 	\lambda_{x_{n-1}}\cdots\lambda_{x_1}\phi^{x_1}_T=&\lambda_{x_{n-1}}\cdots\lambda_{x_2}\phi^{x_2}_T\lambda_{x_1}\\
 	=&\lambda_{x_{n-1}}\cdots\lambda_{x_i}\phi^{x_i}_T\lambda_{x_{i-1}}\cdots\lambda_{x_1}\\
 	=&\phi^{x_n}_T\lambda_{x_{n-1}}\cdots\lambda_{x_1},
 	\end{align*}
 	%\end{equation}
 	as claimed.
 \end{proof}
  \section{Drinfeld Modular Curves and their  Generalizations}\label{Sec:DMC}
 \subsection{Drinfeld Modular Curves}By Lemma \ref{Lem:moduleiso}, for a polynomial $ N \in A $ satisfying  $(T-1,N)=1$,     the kernel $ \Ker({\phi_N })$ is isomorphic to  $(A/(NA))^{\oplus m } $ as an $A$-module. 

\begin{notation}
Denote by $ \mathcal{G} (N;\phi )$ the set of all rank $1$ $N$-torsion submodules $G\subseteq \Ker (\phi_N ) $ (i.e.,  $ G \cong A/(NA) $) such that $\xi(G)=G$ for all $\xi\in\Aut(\bar L/L)$.  	

\end{notation}

Let $N$ and $\mathcal{G}(N;\phi )$ be as above. 
	Two pairs $ (\phi, G) $ and $ (\phi', G') $, where $\phi$ and $\phi' $ are Drinfeld modules of the form \eqref{Eq:type}, $G\in \mathcal{G}(N;\phi )$, and $G'\in \mathcal{G}(N;\phi' )$, are said to be \textit{equivalent}, if there exists an isomorphism $\lambda$ from $\phi$ to $\phi'$ such that $\lambda G=G'$. 
	
	The following    three types of Drinfeld modular curves, all  adapted from those of Elkies \cite{Elkies2001}, can be seen as analogues  to   the classical    modular curves which parameterize elliptic curves  associated with certain  level structures.
	
	\begin{defn} 
		\label{Defn:normalizedElkiesModularCurve}Suppose
		that  $ N \in A $ is a polynomial satisfying  $(T-1,N)=1$ and $T|N$. 	
		
		\begin{enumerate}
			\item 		
		The \textit{Drinfeld modular curve} $ X_{m,j}(N) $ with respect to the polynomial $N$ is the  algebraic curve that parameterizes  equivalent classes of pairs $(\phi,G)$, where $\phi$ is a  Drinfeld module  which is isomorphic to a   normalized one   
		and $G\in  \mathcal{G}(N;\phi )$.
 \item
 	The  \textit{Drinfeld modular curve} $ \dot{X}_{m,j}(N) $  is the algebraic curve which parameterizes pairs  $(\phi,G)$, where $\phi\in \mathcal{D}_{m,j}$ is a normalized Drinfeld module and $G\in \mathcal{G}(N;\phi )$.
 \item 
 The  \textit{Drinfeld modular curve} $ \ddot{X}_{m,j}(N) $  is the algebraic curve which parameterizes triples $(\phi,G,x)$, where $\phi\in\mathcal{D}_{m,j}$ is a normalized Drinfeld module, $G\in \mathcal{G}(N;\phi )$, and $x$   (called marked point) is a nonzero element of $G\cap \Ker(\phi_T)\cap L$ (which is isomorphic to $\F_q$).  
\end{enumerate}
\end{defn}

\begin{remark}\label{Rem:J}
	\begin{enumerate}
		\item In the particular case that $ N = 1 $,  the curve $ X_{m,j}(1) $ coincides with the $ J $-line, $J $ being the coordinate that tells the $J $-invariant of  Drinfeld modules (see 
		Lemma \ref{Lem:Isomorphic}).
		\item For the curve $\ddot{X}_{m,j}(N)$ that parameterizes triples $(\phi,G,x)$, we have $\phi=\phi^x$ (see Notation \ref{Notation:phix}).
		\item The $q^m $-Frobenius morphism of $\ddot{X}_{m,j}(N)$ is given below:  \[(\phi,G,x)\mapsto(\phi',G',x'),\] where
		$x'=x^{q^m}$, $\phi'=\phi^{x'}$,    and $G'= G^{q^m}$. In addition and in a similar fashion,  $q^m $-Frobenius morphisms of curves $\dot{X}_{m,j}(N)$ and ${X}_{m,j}(N)$ can be defined. 
		
	\end{enumerate}
\end{remark}

\subsection{Towers and Galois Coverings}
Let us consider the particularly interesting polynomials  $ N= T ^n $, for   $n= 1,2,\cdots $. There associates three natural     towers   of modular curves. The one formed by $\ddot{X}_{m,j}(T^n)$ is drawn below:
\begin{equation*}%\label{Tow:dXmj}
\begin{tikzcd}
\ddot{X}_{m,j}(T)    &\ddot{X}_{m,j}(T^2)\arrow[l, "p_1"']  &\ddot{X}_{m,j}(T^3) \arrow[l,"p_2"']   &\cdots \arrow[l,"p_3"'] ,
\end{tikzcd}
\end{equation*} 
where   $\{p_n\}_{n\geqslant 1}$ is defined by
\begin{equation*}%\label{Eq:defnpn}
\begin{aligned}
p_n:~\ddot{X}_{m,j}(T^{n+1})&\to\ddot{X}_{m,j}(T^{n})\\
(\phi ,G_{n+1},x_1)&\mapsto(\phi , \phi_TG_{n+1},x_1).
\end{aligned}
\end{equation*} 

The second tower of Drinfeld modular curves  $\dot X_{m,j}(T^n)$ and the third one for  $X_{m,j}(T^n)$ are built similarly.  Moreover, the three towers of curves are organized in the following diagram:
\begin{equation}\label{Eq:threetowers}
\begin{tikzcd}
\ddot{X}_{m,j}(T)\arrow[d,"\pi_1"]   &\ddot{X}_{m,j}(T^2)\arrow[d,"\pi_1"]\arrow[l]  &\ddot{X}_{m,j}(T^3) \arrow[d,"\pi_1"] \arrow[l]   &\cdots \arrow[d,"\pi_1"]\arrow[l] &\ddot{X}_{m,j}(T^n) \arrow[d,"\pi_1"] \arrow[l] &\cdots \arrow[d,"\pi_1"]\arrow[l]\\
\dot{X}_{m,j}(T)\arrow[d,"\pi_2"]   &\dot{X}_{m,j}(T^2)\arrow[d,"\pi_2"]\arrow[l] &\dot{X}_{m,j}(T^3) \arrow[d,"\pi_2"] \arrow[l] &\cdots \arrow[d,"\pi_2"]\arrow[l] &\dot{X}_{m,j}(T^n) \arrow[d,"\pi_2"] \arrow[l] &\cdots \arrow[d,"\pi_2"]\arrow[l]\\
X_{m,j}(T)   &X_{m,j}(T^2)\arrow[l]  &X_{m,j}(T^3)\arrow[l]  &\cdots \arrow[l] &X_{m,j}(T^n) \arrow[l] &\cdots .\arrow[l]
\end{tikzcd}
\end{equation}

{ The vertical  morphisms $\pi_1$ and $\pi_2$  are defined using  their $L$-points as specified below: 
\begin{equation*}
	 \pi_1:\ddot X_{m,j}(T^n)\to  \dot X_{m,j}(T^n),\qquad (\phi,G_n,x_1)\mapsto (\phi,G_n),\end{equation*}
	 {and}
	 \begin{equation*}
\pi_2:\dot X_{m,j}(T^n)\to X_{m,j}(T^n), \qquad (\phi,G_n)\mapsto [(\phi,G_n)].	 
	 \end{equation*} 
	 Let us denote the composition of $\pi_1$ and $\pi_2$
%	 :
%\[\begin{tikzcd}
%\ddot X_{m,j}(T^n)\arrow[r,"\pi_1"] \arrow[rr, bend left,"\pi_3"] & \dot X_{m,j}(T^n)\arrow[r,"\pi_2"] &X_{m,j}(T^n),
%\end{tikzcd}\]
 by
$$
\pi_3:\ddot X_{m,j}(T^n)\to X_{m,j}(T^n),\qquad (\phi,G_n,x_1)\mapsto [(\phi,G_n)].
$$ 

Our  Main Theorem in the introduction  claims that  the three towers above are all recursive. Note that the associated recursive  polynomials coincide with those  investigated by  Bassa \textit{et al.} in \cite{Bassa2015,A.Bassa2014}.}

Let $\ddot{F}_{m,j}^{(n)}$ (resp. $\dot{F}_{m,j}^{(n)}$, $F_{m,j}^{(n)}$) be the function field of $\ddot{X}_{m,j}(T^n)$ (resp. $\dot{X}_{m,j}(T^n)$, $X_{m,j}(T^n)$).  In parallel with  \eqref{Eq:threetowers}, we are able to draw  a diagram of function fields: 
\[
\begin{tikzcd}
\ddot{F}_{m,j}^{(1)} \arrow[r]    &\ddot{F}_{m,j}^{(2)}  \arrow[r] &\ddot{F}_{m,j}^{(3)}\arrow[r]  &\cdots\arrow[r]
&\ddot{F}_{m,j}^{(n)}\arrow[r] &\cdots\\
\dot F_{m,j}^{(1)}\arrow[u]\arrow[r]    &\dot F_{m,j}^{(2)}  \arrow[u] \arrow[r] &\dot F_{m,j}^{(3)} \arrow[u] \arrow[r]  &\cdots \arrow[u]\arrow[r]
&\dot{F}_{m,j}^{(n)}\arrow[u]\arrow[r] &\cdots \\
F_{m,j}^{(1)}\arrow[u]\arrow[r]    &F_{m,j}^{(2)} \arrow[u]   \arrow[r] &F_{m,j}^{(3)}\arrow[u]  \arrow[r]  &\cdots \arrow[u] \arrow[r] &F_{m,j}^{(n)} \arrow[u] \arrow[r] &\cdots .
\end{tikzcd}
\]
In the rest of this section, we establish some facts about    relative degrees of morphisms appeared in Diagram \eqref{Eq:threetowers}. For convenience, we write
$\mathcal{G}_n(\phi):=\mathcal{G}(T^n;\phi)$.
\begin{lem}\label{Lem:pndeg}
	The relative degree of the morphism $p_n:\ddot{X}_{m,j}(T^{n+1})\to \ddot{X}_{m,j}(T^n)$  is $q^{m-1}$. 
\end{lem}

\begin{proof}
	Without loss of generality, we may assume that $ L = \bar{L} $. %(Otherwise one ....????) 
	Let $G_n\in\mathcal{G}_n(\phi)$ be fixed. The lemma is proved if we can show that there are exactly $q^{m-1}$ elements $G_{n+1}\in\mathcal{G}_{n+1}(\phi)$  such that $\phi_T(G_{n+1})=G_n$.
	
	We will   list such $G_{n+1}$ explicitly. First, one can find some  $\mu\in\bar L$ such that $G_n=A\cdot\mu$, because $G_n\cong A/(T^nA)$. Second, consider the set
	$$\mathfrak{S}:=\left\lbrace \nu\in\bar L|\phi_T(\nu)=\mu\right\rbrace.
	$$
	Each $\nu\in \mathfrak{S}$ gives rise to an $A$-module $G_{n+1}=A\cdot \nu\in \mathcal{G}_{n+1}(\phi)$ which certainly satisfies $\phi_T(G_{n+1})=G_n$. It is also easy to see that all solutions $G_{n+1}$ to $\phi_T(G_{n+1})=G_n$ must be of this form.
	
	Finally,   
	$\nu$ and $\nu'\in \mathfrak{S}$ give rise to the same $G_{n+1}$ if and only if $\nu-\nu'\in G_{n+1}\cap\Ker(\phi_T)\cong A/(TA)$. 	
	Therefore,  the number of such $A$-modules $G_{n+1}$ is computed by: 
	$$\frac{\# \mathfrak{S}}{ \# (A/(TA))}=\frac{q^m}{q}=q^{m-1} .$$ This completes the proof.
\end{proof}

% SO FAR !!!

For $\mu\in \F_{q^m}^*$, there associates an automorphism on  $\ddot{X}_{m,j}(T^n)$  defined by
\begin{equation}\label{Eq:mudefinedby}
 (\phi,G_n,x_1)\mapsto (\phi^{\mu x_1},\mu G_n, \mu x_1).\end{equation}
 This automorphism is also denoted by  $\mu$.
\begin{lem}\label{Lem:RelDDD}
	\begin{enumerate}
\item The automorphism $\mu$ is compatible with the covering $\pi_3$,   i.e., the following diagram 
\[\begin{tikzcd}
\ddot{X}_{m,j}(T^n) \arrow[dr,"\pi_3"] \arrow[rr,"\mu"] & & \ddot{X}_{m,j}(T^n)\arrow[dl,"\pi_3"']\\
&X_{m,j}(T^n)
\end{tikzcd}
\]
of algebraic curves is commutative. Moreover, if $\mu\in\F_q^*$, then $\mu$ is compatible with $\pi_1$.
\item The morphism $\pi_1:~\ddot{X}_{m,j}(T^n) \to \dot{X}_{m,j}(T^n)$ is a Galois covering whose Galois group  is isomorphic to the multiplicative group $\F_q^*$.
\item 
The morphism $\pi_3:~\ddot{X}_{m,j}(T^n) \to   X_{m,j}(T^n)$ is a Galois covering whose Galois group is isomorphic to the multiplicative group $\F_{q^m}^*$. 	\end{enumerate}
\end{lem}
\begin{proof} Note that the statement in part (2) follows  immediately from that of part (3). So we only need to prove parts (1) and (3).
\begin{enumerate}
\item[(1)] Let $\phi$ be the Drinfeld module with $$\phi_T=-\tau^m+g_j\tau^j+1,$$ and hence we have
$$\phi_T^{\mu x_1}=-\tau^m+(\mu^{1-q^j})g_j\tau^j+1 ~\mbox{ and }~\mu \phi_T=\phi_T^{\mu x_1}\mu,$$  by direct calculation. It implies that $[(\phi,G_n)]=[(\phi^{\mu x_1},\mu G_n)]$, i.e. the first statement of part (1). The second statement follows by observing that   $\phi=\phi^{\mu x_1}$ and $G_n=\mu G_n$, if $\mu\in\F_q^*$. 
 %By part (1), it is enough to show that the degree of $\pi_1$ is equal to $(q-1)$. This is evident as $0\neq x_1\in\phi_{T^{n-1}}G_n\cong \F_q$.
\item[(3)] %By part (1), we know that $\phi^{\mu x_1}$ is isomorphic to $\phi$. 
According to part (1), we only need to show that the degree of $\pi_3$ equals $(q^m-1)$.  

First, consider the situation that   $g_j\neq 0$ in the expression  of $\phi_T$.  Since $j$ and $m$ are coprime, we have $\F_{q^m}^{q^j-1}=\F_{q^m}^{q-1}$. Hence the number of Drinfeld modules of the form   $\phi^{\mu x_1}$, for $\mu\in\F_{q^m}^*$, is equal to $\frac{q^m-1}{q-1}$. 
In the meantime,    the number of nonzero elements in  $\phi_{T^{n-1}}G_n$ is $(q-1)$. Thus the number of preimages of $[(\phi,G_n)]$ under $\pi_3$ is   obtained: $\frac{q^m-1}{q-1}(q-1)=q^m-1$. 

Second, if $g_j=0 $, i.e.,  $\phi_T=-\tau^m+1$, then  $\mu\phi=\phi\mu$, for all $\mu\in \F^*_{q^m}$, and hence $[(\phi,G_n)]=[(\phi,\mu G_n)]$. Moreover, $(\phi,G_n,x_1)=(\phi,\mu G_n,x_1)$ if and only if $\mu\in \F^*_{q}$. Again, by the fact that the number of nonzero elements in  $\phi_{T^{n-1}}G_n$ is $(q-1)$, we get the the number of preimages of $[(\phi,G_n)]$ under $\pi_3$, which is $\frac{q^m-1}{q-1}(q-1)=q^m-1$. 

This shows that
the degree of $\pi_3$ is $q^m-1$ and the assertion is thus confirmed.
\end{enumerate}
\end{proof}

\begin{cor}\label{Cor:DXTn}
	All  horizontal morphisms  $$\ddot{X}_{m,j}(T^{n+1})\to \ddot{X}_{m,j}(T^n),\,~ \dot{X}_{m,j}(T^{n+1})\to \dot{X}_{m,j}(T^n),\,~  \mbox{and }~   {X}_{m,j}(T^{n+1})\to  {X}_{m,j}(T^n)$$ in Diagram \eqref{Eq:threetowers}, have the same relative degree  $q^{m-1}$.	
\end{cor}
\begin{proof}
	The conclusion follows directly by Lemmas  \ref{Lem:pndeg}, \ref{Lem:RelDDD}, and the commutativity of Diagram \eqref{Eq:threetowers}.
\end{proof}

\subsection{Modular Curves $\ddot{M}_{m,j}(T^n)$}

%Galois towers over non-prime finite fields, Acta Arith 164 (2014), no. 2, 163–179. MR3224833
%Towers of function fields over non-prime finite fields, Mosc. Math. J. 15 (2015), no. 1, 1–29

In subsequent analysis, we assume that $k$ ($=m-j$) is not divided by  the characteristic $p$ of $\F_q$.  
Let $A_k:=\F_{q^k}[T]$ be the obvious  extension of the ring $A=\F_q[T]$, and moreover, we treat $\bar{L}$ as 
an $A_k$-field in an obvious way. 
Set
\[F(T):=1-(1-T)^k=T\cdot f(T)\in A,\]
where
\[f(T):=\sum_{i=1}^k\binom{k}{i}(-T)^{i-1}.\]
Evidently, $f(0)=k\neq 0$ and $(T,f(T))=1$.

For a normalized Drinfeld module $\phi$ as in \eqref{Eq:phiTandphiT'}, there associates another Drinfeld module 
\begin{equation*}%\label{Eq:Phi}
	\Phi:~A\to L\{\tau\},\qquad T\mapsto \phi_{F(T)}
\end{equation*}
over $L$. In other words,  
%\begin{equation}\label{41}
\begin{align*}
\Phi_T &=\phi_{F(T)} =\phi_{1-(1-T)^k} =1-\phi_{(1-T)^k}=1-(\tau^m-g_j\tau^j)^k\\
&=1-\sum_{i=0}^{k}\left[ k \atop i\right]\tau^{mi}\cdot\tau^{j{(k-i)}}\\
&=1-\sum_{i=0}^{k}\left[ k \atop i\right]\tau^{(i+j)k}, %:=1-G(\sigma).
\end{align*}
%\end{equation}
 where the coefficients are  defined by  
 %\color[rgb]{1,0,0}
\begin{equation*}
\begin{aligned}
\left[ k \atop k\right]=&1; \quad \left[ k \atop 0\right]=(-1)^{k}g_j^{\tr_k(q^j)}; \quad \text{and iteratively,}\\
\left[ k \atop i\right]=&{\left[ k-1 \atop i-1\right]}^{q^m}-g_j{\left[ k-1 \atop i\right]}^{q^j},\quad  \text{for $1\leqslant i\leqslant k-1$}.\\
\end{aligned}
\end{equation*}	
The $\tau$-twisted polynomial $\Phi_T$ can be regarded as a $\tau^k$-twisted polynomial.
%The endomorphism   $\tau^k$ is $q^k$-Frobenius. 
Therefore,   one can alternatively treat $\Phi$ as a Drinfeld module over the $A_k$-field $\bar{L}$ of characteristic $(T-1)$. Recall that $\phi$ gives rise to an $A$-module structure on $\bar L$. Similarly, $\Phi$ gives rise to an $A_k$-module structure on $\bar L$. According to Lemma \ref{Lem:moduleiso}, we have 
\begin{equation}\label{Eq:kerPhi_Ak}
	\Ker(\Phi_{T^n})\cong (A_k/(T^nA_k))^{\oplus m},\qquad\mbox{as an $A_k$-submodule  of $\bar L$},
\end{equation}
and
\begin{equation*}%\label{Eq:kerPhi_A}
	\Ker(\Phi_{T^n})\cong (A/ (T^n f(T)^nA))^{\oplus m},\qquad\mbox{as an $A$-submodule  of $\bar L$}.
\end{equation*}

\begin{notation}\label{Notation:mathcalE} Let $\Phi$ be the Drinfeld module arising from $\phi$  explained as above. 
Denote by $\mathcal{E}_n(\phi)$   the set consisting of $\F_{q^k}$-vector spaces $E_n\subseteq \Ker(\Phi_{T^n})$, such that \begin{itemize}
	\item[1)] $E_n$ is stable under     $\Aut(\bar L/L)$,   i.e., $\xi(E_n)=E_n$ for all $\xi\in\Aut(\bar L/L)$;
	\item[2)] $	E_n\cong A/(F(T)^nA)$  {as an $A$-submodule of } $\bar L$; and
	\item[3)] $E_n\cong A_k/(T^nA_k)$  as an $A_k$-submodule of $ \bar L$.
	\end{itemize}
\end{notation}

\begin{defn}
	The \textit{twisted Drinfeld modular curve} $\ddot{M}_{m,j}(T^n)$ is the algebraic curve that parameterizes triples $(\phi,E_n,x)$, where $\phi\in \mathcal{D}_{m,j}$,  $E_n\in \mathcal{E}_n(\phi)$,  
	and  $x\in \phi_{T^{n-1} f(T)^n}E_n\cap L$ is a nonzero marked point.  
\end{defn}

\begin{remark}
	If $j=m-1$ $\left(\text{or } k=1\right)$, then $\Phi=\phi$ and the  {twisted Drinfeld modular curve}  $\ddot M_{m,{m-1}}(T^n)$   coincides with the  Drinfeld modular curve $\ddot X_{m,{m-1}}(T^n)$ (see Definition \ref{Defn:normalizedElkiesModularCurve}).  
\end{remark}

The following key theorem  is needed. 
\begin{thm}\label{Thm:IsomphicXM}
	The curves $\ddot{M}_{m,j}(T^n)$ and $\ddot{X}_{m,j}(T^n)$  are isomorphic over $\F_{q^m}$.
\end{thm} 
A direct consequence of this theorem is   that the function field of 
twisted Drinfeld modular curves and that of normalized Drinfeld modular curves
are one and the same. 
Before we come to the proof of this theorem, let us establish some useful  lemmas.
\begin{lem}\label{lem:xbusiness}
Let   $x\in \bar L$ be nonzero. 
Endow $\bar L$ with the $A$-module  structure induced by the Drinfeld module   $\phi=\phi^x$. We have
\begin{enumerate}
	\item 
	The $\F_q$-vector space $\F_{q^k}\cdot x$ is an $A$-submodule of $\bar L$; 
	\item The annihilator ideal of $\F_{q^k}\cdot x$ is generated by $F(T)$;
	\item The $A$-modules $\F_{q^k}\cdot x$ and $A/(F(T)A)$ are isomorphic.
\end{enumerate}  \end{lem}
\begin{proof} (1)
The $A$-module structure of $\F_{q^k}\cdot x$ is presented by
	\begin{align*}
	(1-T)\cdot (\mu x):=&\phi_{1-T}(\mu x)=(\tau^m-g(x)\tau^j)(\mu x)\\
	=&\mu^{q^m}x^{q^m}-g(x)\mu^{q^j}x^{q^j}=-\mu^{q^j}\phi_T(x)+\mu^{q^j}x=\mu^{q^j}x,
	\end{align*}
	for all $\mu\in\F_{q^k}$. 
	
	(2) Let $\{\rho, \rho^q,\ldots,\rho^{q^{k-1}}\}$ be a normal basis of $\F_{q^k}/{\F_q}$. Since $k$ and $j$ are coprime, the action by $(1-T)$  on $\F_{q^k}\cdot x$ is	 a circulant permutation to this basis. Thus  the minimal polynomial of $(1-T)$  equals   
$(\lambda^k-1)$.  
This means that the annihilator ideal of $\F_{q^k}\cdot x$  is generated by  $(1-T)^k-1=-F(T)$. 

(3) By construction, $\F_{q^k}\cdot x$ and $A\cdot \rho x$ are identical. It follows that $\F_{q^k}\cdot x\cong A/{(F(T)A)}$. 
\end{proof}
Recall the notation $  \mathcal{E}_n(\phi)$ that we introduced earlier in Notation \ref{Notation:mathcalE}.
\begin{lem}\label{Lem:welldefineness}
	For $G_n\in \mathcal{G}_n(\phi)$, the $\F_{q^k}$-subspace of  $\bar L$ spanned by $G_n$, denoted by $\F_{q^k}\left\langle G_n\right\rangle$, belongs to $  \mathcal{E}_n(\phi)$. 
	\end{lem}
\begin{proof}According to the definition of $\mathcal{E}_n(\phi)$, we divide the proof  into three parts.
	\begin{enumerate}
		\item For any $\xi\in\Aut(\bar L/L)$,  we have $\xi(G_n)=G_n$, by definition of $\mathcal{G}_n(\phi)$.
		For $ \sum_i \mu_i g_i \in \F_{q^k}\left\langle G_n\right\rangle $, we see that 
		\[\xi(\sum_i \mu_i g_i)=\sum_i \xi(\mu_i)\xi(g_i) \in \F_{q^k}\left\langle G_n\right\rangle,\]
		which means that  $\Aut(\bar L/L)$ preserves  $\F_{q^k}\left\langle G_n\right\rangle$.
		\item Regarding the $A_k$-module structure, 		
		we first examine that $\F_{q^k}\left\langle G_n\right\rangle$ is closed under the $\Phi_T$-action. For $ \sum_i \mu_i g_i \in \F_{q^k}\left\langle G_n\right\rangle $, we have
		$$
		%T\cdot(\sum_i \mu_i g_i)=
		\Phi_{T}(\sum_i \mu_i g_i)= \sum_i \mu_i\phi_{F(T)}(g_i)\in \F_{q^k}\left\langle G_n\right\rangle .
		$$ 
		
		We next show that $\F_{q^k}\left\langle G_n\right\rangle\cong A_k /(T^n A_k)$. This fact is due to the following  observations:
		\begin{enumerate}
			\item The subspace $ \F_{q^k}\left\langle G_n\right\rangle$ is contained in $\Ker (\Phi_{T^n}) $. This is easy as
					\[ \Phi_{T^n}(\sum_i \mu_i g_i )= \sum_i \mu_i \Phi_{T^n}(g_i)= \sum_i \mu_i \phi_{f(T)^n}\phi_{T^n}(g_i)=0. \] 
			\item We can find  $ u \in G_n  $ such that $ \Phi_{T^{n-1}}(u) \not = 0 $. In fact,  we have
			\[ \Phi_{T^{n-1}} (G_n) = \phi_{T^{n-1}} \phi_{f(T)^{n-1}}(G_n) 
			= \phi_{T^{n-1}} (G_n ) \not = \{ 0 \},
			\]
			where we used the fact that $G_n\cong A/(T^nA)$ and $(T,f(T))=1$.

			\item  We show that $A_k/(T^n A_k)\cong A_k \cdot u $. In fact, by (a) and the isomorphism in \eqref{Eq:kerPhi_Ak}, $$u\in 
			\Ker (\Phi_{T^n}) \cong (A_k /(T^n A_k))^{\oplus m}.$$ 
			By (b), one has $\Ann(u)=(T^n)$, and thus $A_k\cdot u=A_k/\Ann(u)=A_k/(T^n A_k)$.
			 
			\item As $G_n\cong A/(T^nA)$, we have $\dim_{\F_q}(G_n)=n$.  Comparing the $F_{q^k}$-dimensions of $A_k\cdot u\subseteq \F_{q^k}\left\langle G_n\right\rangle$,   we see that they must be equal.
		\end{enumerate}

		\item Regarding the $A$-module structure of $G_n$, we examine that $\F_{q^k}\left\langle G_n\right\rangle$ admits a  $\phi_T$-action. In fact, for $ \sum_i \mu_i g_i \in \F_{q^k}\left\langle G_n\right\rangle $, we have
		\[\phi_T(\sum_i \mu_i g_i)=\sum_i\mu_i^{q^j}\left(\phi_T(g_i)+(u-u^{q^j})g_i\right)\in \F_{q^k}\left\langle G_n\right\rangle.\]
		
		We now show that $\F_{q^k}\left\langle G_n\right\rangle\cong A/(F(T)^nA)$ as $A$-modules. This is accomplished   following three steps.
		\begin{enumerate}
			\item By what we concluded in the second part,  we have $A_k/(TA_k)\cong \Phi_{T^{n-1}}(\F_{q^k}\left\langle G_n \right\rangle )$ as $A_k$-modules. This implies that $\Phi_{T^{n-1}}(\F_{q^k}\left\langle G_n \right\rangle )=\F_{q^k}\cdot x $, for some $x\in \bar L$.
			\item According to Lemma \ref{lem:xbusiness}, $$\phi_{F(T)^{n-1}}(\F_{q^k}\left\langle G_n \right\rangle)=\Phi_{T^{n-1}}(\F_{q^k}\left\langle G_n \right\rangle )=A\cdot\rho x\cong A/(F(T)A),$$ as $A$-modules.  And we are able to find some $v\in \F_{q^k}\left\langle G_n \right\rangle $ such that $$\phi_{F(T)^{n-1}}(v)=\Phi_{T^{n-1}}(v)=\rho x.$$ Because $\Ann(\rho x)=(F(T))$, we have $\Ann(v)=(F(T)^n)$ and $$A \cdot v\cong A/(\Ann(v))=A/(F(T)^n A).$$
			\item It follows from the previous argument that  the  dimension  of the $\F_q$-vector space  $A\cdot v$ is $q^{nk}$, the same as that of $\F_{q^k}\left\langle G_n \right\rangle$. So one must have $A\cdot v=\F_{q^k}\left\langle G_n \right\rangle$. 
					\end{enumerate}
		 
	\end{enumerate} 
This completes the proof.
\end{proof} 

We are now ready to give the 
\begin{proof}[Proof of Theorem \ref{Thm:IsomphicXM}]  We construct a map
	 \begin{align*}
	\alpha:\ \ddot{M}_{m,j}(T^n)&\to \ddot{X}_{m,j}(T^n),\\
	(\phi,E_n,x)&\mapsto (\phi,\phi_{f(T)^n}E_n,x).
	\end{align*}
 	Here we have used the fact that \[A/\left({T^nA}\right)\cong \phi_{f(T)^n}E_n\in \mathcal{G}_n(\phi),  \]
	which is due to the definition of $E_n$. 	
	In the mean time, we construct	 	\begin{align*}
	\beta:\ \ddot{X}_{m,j}(T^n)&\to \ddot{M}_{m,j}(T^n),\\
	(\phi,G_n,x)&\mapsto (\phi,\F_{q^k}\left\langle G_n\right\rangle,x).
	\end{align*}
	 
By Lemma \ref{Lem:welldefineness}, the map  $\beta$ is well-defined. 	We now prove that $\alpha$ and $\beta$ are mutually inverse maps.
	\begin{enumerate}
	\item	
	First, we show that $\alpha\circ \beta =\id$, which amounts to the following identity
	\begin{equation}\label{Eq:phif1}
	\phi_{f(T)^n}(\F_{q^k}\left\langle G_n \right\rangle)=G_n,\qquad \mbox{for all}~ G_n\in \mathcal{G}_n(\phi).
	\end{equation}
	
	In fact, by Lemma \ref{Lem:welldefineness}, we know that $$\F_{q^k}\left\langle G_n \right\rangle\cong A/(F(T)^nA)\cong A/( f(T)^nT^nA), $$ and hence $$\phi_{f(T)^n}(\F_{q^k}\left\langle G_n \right\rangle)\cong A/\left({T^nA}\right).$$ 	It implies that
	\[G_n=\phi_{f(T)^n}G_n\subseteq\phi_{f(T)^n}(\F_{q^k}\left\langle G_n \right\rangle) .\]
 By the fact that $$\#{G_n}=\#{\bigl(A/\left({T^nA}\right)\bigr)}=\#
 \bigl(\phi_{f(T)^n}(\F_{q^k}\left\langle G_n \right\rangle)\bigr),$$  
 we proved the desired equality \eqref{Eq:phif1}.	
\item	
	Second, we show that $\beta\circ\alpha=\id$, or 
	$$E_n=\F_{q^k}\left\langle \phi_{f(T)^n}E_n\right\rangle,\qquad \mbox{for all }~ E_n\in \mathcal{E}_n(\phi).$$
	
	 In fact,   $\phi_{f(T)^n}E_n$ is isomorphic to $A/(T^nA)$ (as an $A$-submodule of $\bar L$) by its definition. Using Lemma \ref{Lem:welldefineness}, we get 
	\[
	\F_{q^k}\left\langle \phi_{f(T)^n}E_n\right\rangle\cong A_k/(T^nA_k).
	\]
	Counting cardinalities of the two sides of $$\F_{q^k}\left\langle \phi_{f(T)^n}E_n\right\rangle \subseteq E_n,$$ we
	see that they must be equal (to the same number $q^{kn}$). 
The proof is thus completed. 
	\end{enumerate}
\end{proof}
\section{Proof of Generalized Elkies' Theorem}\label{Sec:Main part}
This section is denoted to  proving   the Main Theorem declared in  the introduction part.  
\subsection{Part A}
% Recall the covering morphisms $p_n$ defined in  Equation (\ref{Eq:defnpn}).
Recall that in Notation \ref{Notation:phix}, we introduced the Drinfeld module  $\phi^{x }$  such that $\phi^{x }_T(x)=0$, where $0\neq x\in \bar L$. 
Recall also that in Notation \ref{Notation:mathcalE}, we defined the set 
\[\mathcal E_n(\phi):=\left\lbrace E_n\subseteq\Ker\left(\phi_{F(T)^n}\right)|E_n\cong A/(F(T)^n A) \text{ as $A$-modules}, 
\atop \text{ and }\ E_n\cong A_k/(T^n A_k) \text{ as $A_k$-modules}\right\rbrace.\]

For $0\neq x_1\in \bar L$, let us define the set

\[\mathcal E_n^*(\phi^{x_1}):=\left\lbrace E_n\in\mathcal E_n(\phi^{x_1})|x_1\in %\phi^{x_1}_{T^{n-1}}\phi^{x_1}_{f(T)^n}E_n =
\Phi^{x_1}_{T^{n-1}}\phi^{x_1}_{f(T)}E_n \right\rbrace.\]

 Next, we consider another set      $$X_n:=\{(x_2,\ldots, x_n)\in (\bar L)^{n-1}~|~ Q_{x_1}(x_2)=x_1,~Q_{x_2}(x_3)=x_2,~\cdots,~Q_{x_{n-1}}(x_n)=x_{n-1}\}.$$
 Apparently, %together with $x_1$, 
 $(x_2,\ldots, x_n)\in X_n$ is subject to Equation \eqref{Eq:A,m,j}.
 
 We need a preparatory theorem.
 \begin{thm}\label{Thm:FMT2}
 	Let $x_1\in \bar L$ be nonzero, the sets $\mathcal E_n^*(\phi^{x_1})$  and   $X_n$  as above. For  $n\geqslant 2$,
 	there is a one-to-one correspondence between   $\mathcal E_n^*(\phi^{x_1})$  and   $X_n$ as sets.    
 	 \end{thm}
  \begin{proof}
  	 We will establish the following fact --- For each step $n  \geqslant 2$,  
  	there is a bijective map 
  	\[ \Theta_n:~\mathcal E_n^*(\phi^{x_1})\to X_n. \]Moreover, the correspondence $ E_n\in \mathcal E_n^*(\phi^{x_1}) ~\mapsto ~(x_2,\ldots,x_n)\in X_n$ is characterized by the following conditions:
  	\begin{itemize} 
  		\item[(I)]As a submodule in $\bar L$, we have \begin{equation}\label{Eqt:E_n=...}
  		E_n=\Ker(\lambda_{x_{n}}\cdots\lambda_{x_2}\lambda_{x_1});
  		\end{equation}
  		\item[(II)] The set  $ H_n:=E_n\cap (\Phi_{T^{n-1}})^{-1}( x_1)$ is non-empty, and for each $h\in H_n$, one has $$x_n =\lambda_{x_{n-1}}\cdots\lambda_{x_1}(h).$$
  		%	\item[(3)]
  	\end{itemize}
  	
  	We prove by induction on $n$, and start from $n=2$. 
  	
  	\begin{enumerate}%[\textcircled{}] 
  		\item 
  		The existence of $h\in H_2$ is due to the definition of% $\mathcal E_2^*(\phi^{x_1})$
  		\[\mathcal E_2^*(\phi^{x_1}):=\left\lbrace E_2\in\mathcal E_2(\phi^{x_1})|x_1\in \Phi^{x_1}_{T }\phi^{x_1}_{f(T)}E_2 \right\rbrace.\]
  		
  		\item  The value $\lambda_{x_1}(h)$ does not depend on the choice of $h$.  In fact, if  $ h'\in E_2$ satisfies $\Phi^{x_1}_T(h')=x_1$, then $h-h'\in \Ker(\Phi_T)\cap E_2$.   So we only need to show that $\Ker(\Phi_T)\cap E_2=\Ker (\lambda_{x_1})$.
  		
  		The fact that    $\Ker (\lambda_{x_1})\subseteq \Ker(\Phi_T)\cap E_2$ is obvious. Also note that $\#(\Ker (\lambda_{x_1}))=q^k$. Let us show that $\#(\Ker(\Phi_T)\cap E_2)$ is also $q^k$. This can be examined by observing that $E_2\cong A_k/(T^2 A_k)$, and   $\Ker(\Phi_T)\cap E_2\cong A_k/(TA_k) $.  Hence $\Ker (\lambda_{x_1}) $ and $ \Ker(\Phi_T)\cap E_2$ must be identical.
  		
  		\item  The value
  		$x_2:=\lambda_{x_1}(h)$ belongs to $X_2$, i.e., 
  		\begin{equation}
  		\label{Eq:temp1}Q_{x_1}(x_2)=x_1.
  		\end{equation} 
  		We will use  the equality  $\phi^{x_2}_{f(T)}(x_2)=f(0)x_2=kx_2$ to prove this equation. Let us do some computation:  
  		\begin{equation*}%\label{Eq:E2xy}
  		\begin{aligned}
  		\mbox{LHS of Equation \eqref{Eq:temp1}}  =&Q_{x_1}\left(\frac{1}{k}\phi^{x_2}_{f(T)}(x_2)\right)\\
  		=&\frac{1}{k}Q_{x_1}\left(\phi^{x_2}_{f(T)}\lambda_{x_1}(h)\right)\\
  		=&\frac{1}{k}Q_{x_1}\left(\lambda_{x_1}\phi^{x_1}_{f(T)}(h)\right)\quad\text{ (by Theorem \ref{Thm:Isolamx})}\\
  		=&\frac{1}{k}\eta_{x_1}\phi^{x_1}_T\phi^{x_1}_{f(T)}(h)\quad\text{ (by Lemma \ref{Lem:etaphiQlam})}\\
  		=&\frac{1}{k}\eta_{x_1}(\Phi^{x_1}_T(h))=\frac{1}{k}\eta_{x_1}(x_1) = \mbox{RHS of Equation \eqref{Eq:temp1}}.
  		\end{aligned}
  		\end{equation*}
  		The last step is due to the definition of $\eta_x$.
  	\item By the above facts,  it is eligible to   build a map $\Theta_2:~\mathcal{E}_2^*(\phi^{x_1})\to X_2$ 
  	by setting
  	$$
  	\Theta_2(E_2):=\lambda_{x_1}(h),\quad \mbox{ where } h\in H_2.
  	$$ We will show that
  	$\Theta_2:~\mathcal{E}_2^*(\phi^{x_1})\to X_2$ is a bijection.
  	
  	\item Let us show that \begin{equation}
  	\label{Eq:temp2} E_2 =\Ker(\lambda_{x_2}\lambda_{x_1}),\quad \mbox{ if } x_2=\Theta_2(E_2).
  	\end{equation} 	   Indeed, if we take the afore mentioned $h\in E_2$, such that $x_2=\Theta_2(E_2)=\lambda_{x_1}(h)$, then $$\F_{q^k}\left\langle x_1,h\right\rangle \subseteq E_2.$$
Moreover,  by the choice of $h$,   one has $\F_{q^k}\left\langle x_1,h\right\rangle\subseteq \Ker(\lambda_{x_2}\lambda_{x_1}) $. Consequently, we have
$$\F_{q^k}\left\langle x_1,h\right\rangle \subseteq E_2 \cap \Ker(\lambda_{x_2}\lambda_{x_1}) .$$

As   $\F_{q^k}$-vector spaces,  $\F_{q^k}\left\langle x_1,h\right\rangle$, $E_2$, and $\Ker(\lambda_{x_2}\lambda_{x_1})$ are all $2$-dimensional, and hence they must be one and the same. This proves Relation \eqref{Eq:temp2}.

  	\item 	
  	Next, we proceed to show that $\Theta_2:~ \mathcal{E}_2^*(\phi^{x_1})\to X_2$ is a bijection. The fact that $\Theta_2$ is  injective is implied by Relation \eqref{Eq:temp2}. So we only need to show that the two sets $\mathcal{E}_2^*(\phi^{x_1})$ and $X_2$ have the same cardinality: 
  	\begin{itemize}
  		\item [(a)] The number $\#(X_2)=q^{m-1}$, because the twisted degree of $Q_{x_1}$ is $(m-1)$.
  		\item [(b)] Elements in $\mathcal{E}_2^*(\phi^{x_1})$ are  solutions to the relation 
  		$$
  		\Phi_T(E_2)=\F_{q^k}\cdot x_1,\qquad E_2\in \mathcal E_2(\phi^{x_1}).
  		$$
  		In other words, the triple $(\phi,E_2,x_1)$ is the preimage of $(\phi,\F_{q^k}\cdot x_1,x_1)$ under $p'_1$, the right arrow in the following commutative diagram of algebraic curves:
  		$$
  		\xymatrix{
  			\ddot{X}_{m,j}(T^2)  \ar[d]_{p_1} \ar[r]^{\cong }
  			& \ddot{M}_{m,j}(T^2)  \ar[d]^{p'_1}  \\
  			\ddot{X}_{m,j}(T )    \ar[r]_{\cong }
  			& \ddot{M}_{m,j}(T)           .   }
  		$$
  		The two horizontal isomorphisms $\cong$ are due to Theorem \ref{Thm:IsomphicXM}. By Lemma \ref{Lem:pndeg}, the degree of the left $p_1$ is $q^{m-1}$,  and so is the right $p'_1$. 
  		This confirms the fact that the number of such subspaces $E_2$ is   $q^{m-1}$.
  		
  	\end{itemize}

  	\end{enumerate}

  	This completes  proof of the $n=2$ case. 
  	
  	Suppose that the statement is true for the $(n-1)$ case --- We have a bijective map $\Theta_{n-1}:~\mathcal E_{n-1}^*(\phi^{x_1})\to X_{n-1}$ satisfying conditions (I) and (II) with all $n$ replaced by $(n-1)$.
  	
  	\begin{enumerate}
  		\item[(7)] Given $E_n\in \mathcal E_n^*(\phi^{x_1})$, we set  $E_{n-1}:=\Phi_T(E_n)\in\mathcal E_{n-1}^*(\phi^{x_1})$. By   induction assumption, we get the image  $$\Theta_{n-1}(E_{n-1}):=(x_2,\ldots,x_{n-1})\in X_{n-1} ,$$ such that $E_{n-1}=\Ker(\lambda_{x_{n-1}}\cdots\lambda_{x_1})$ and 
  		$x_{n-1} =\lambda_{x_{n-2}}\cdots\lambda_{x_1}(h)$, for all  
  		$h\in H_{n-1}:=E_{n-1}\cap (\Phi_{T^{n-2}})^{-1}( x_1)$.

  		\item[(8)] The fact that $H_n$ is non-empty is due to the definition of $E_n$. Take any   $h\in H_n$. We define $$x_n:=\lambda_{x_{n-1}}\cdots\lambda_{x_1}(h)$$ and
  		$$
  		\Theta_n(E_n):=(x_2,\ldots,x_{n-1},x_n). 
  		$$
  		
  		Of course, we need to verify that
  		$x_n$ does not depend on the choice of $h$. This is completely analogous to what we did in the previous Step (2)   and thus omitted.

  		  	We  also need to show that the data $(x_2,\ldots,x_n)$ belongs to $X_n$, which amounts to the   verification of $Q_{x_{n-1}}(x_n)=x_{n-1}$. Indeed, this  is a routine job done as below.
  		\begin{equation*}
  		\begin{aligned}
  		Q_{x_{n-1}}(x_n)=&Q_{x_{n-1}}\frac{1}{k}\phi^{x_n}_{f(T)}(x_n)\\
  		=&\frac{1}{k}Q_{x_{n-1}}\left(\phi^{x_n}_{f(T)}\lambda_{x_{n-1}}\cdots\lambda_{x_1}(h)\right)\\
  		=&\frac{1}{k}Q_{x_{n-1}}\lambda_{x_{n-1}}\cdots\lambda_{x_1}\phi_{f(T)}(h)\quad\text{ (by Corollary \ref{Cor:isolmdphiPhi})}\\
  		=&\frac{1}{k}\eta_{x_{n-1}}\phi^{x_{n-1}}_T\lambda_{x_{n-2}}\cdots\lambda_{x_1}\phi_{f(T)}(h)\quad\text{ (by Lemma \ref{Lem:etaphiQlam})}\\
  		=&\frac{1}{k}\eta_{x_{n-1}}\lambda_{x_{n-2}}\cdots\lambda_{x_1}\phi_T\phi_{f(T)}(h)\quad\text{ (by Corollary \ref{Cor:isolmdphiPhi})}\\
  		=&\frac{1}{k}\eta_{x_{n-1}}\lambda_{x_{n-2}}\cdots\lambda_{x_1}(\Phi_T(h))\\
  		=&\frac{1}{k}\eta_{x_{n-1}}(x_{n-1})\quad\text{ (by  $\Phi_T(h)\in E_{n-1}$ and $\Phi_{T^{n-2}}(\Phi_T(h))=x_1$)}\\
  		=&x_{n-1}.
  		\end{aligned}
  		\end{equation*}
  		\item[(9)] So far we have constructed  the map $\Theta_n$. 
  		Moreover, Relation \eqref{Eqt:E_n=...}   can be similarly approached as that of Step (5), and thus we   omit the details of verification.  Note that \eqref{Eqt:E_n=...} implies that $\Theta_n$ is injective.
  		
  		\item[(10)]	Finally,   from Lemma  \ref{Lem:pndeg} and Theorem \ref{Thm:IsomphicXM}, we are able to derive the fact that 
  		\[\#(\mathcal E_n^*(\phi^{x_1}))=\#(X_n)=q^{(m-1)(n-1)},\]
  		which forces $\Theta_n$ to be   bijective.  
  	\end{enumerate}
  	The proof is thus  completed.

  \end{proof} 

Now we are ready  to prove  part A of the Main  Theorem. Based on Theorem \ref{Thm:IsomphicXM}, it suffices to determine the function field  of   $\ddot{M}_{m,j}(T^n)$.
\begin{proof}[Proof of  part A]	
We  consider  $\bar L$-points of  curves $\ddot{M}_{m,j}(T^n)$.	
	\begin{enumerate}[1.]
		\item Case $n=1$.

		Take a geometric point $(\phi,E_1,x_1)$ of $\ddot{M}_{m,j}(T)$. It follows from $\phi_T(x_1)=0$
		that $\phi=\phi^{x_1}$. 
		%	\[\phi_T=-\tau^m+(x^{q^m-q^j}-x^{1-q^j})\tau^j+1.\] 
		Since $0\neq x_1\in E_1$ and $E_1$ is a  $1$-dimensional $\F_{q^k}$-vector space, we get
		\[E_1=\Ker(\lambda_{x_1}).\]
		It means that $(\phi,E_1,x_1)$ is   determined by $x_1$. Hence the function field $\ddot{F}_{m,j}^{(1)}$ of  $\ddot{M}_{m,j}(T)$ is equal to $\F_{q^m}(x_1)$.
		
		\item Case  $n\geqslant2$.

		Let $0\neq x_1\in\bar L$ be fixed. By definition,  $\bar L$-points $(\phi,E_n,x_1)$ of $\ddot{M}_{m,j}(T^n)$ are
		in  one-to-one correspondence to $E_n$ in $\mathcal{E}_n^*(\phi^{x_1})$. Then by the previous Theorem  \ref{Thm:FMT2}, 
		the function field $\ddot{F}_{m,j}^{(n)}/{\ddot{F}_{m,j}^{(1)}}$ is generated by variables $x_2,\ldots,x_n$ satisfying a sequence of equations $$Q_{x_1}(x_{2})=x_1 , ~ \ldots , ~ Q_{x_{i}}(x_{i+1})=x_{i},~\ldots,Q_{x_{n-1}}(x_{n})=x_{n-1} .$$ One obtains  relations in Equation  \eqref{Eq:A,m,j} immediately.

	\end{enumerate}
\end{proof}
\subsection{Part B}	
For $\mu\in \F_q^*$, recall   the automorphism $\mu$ on $\ddot{X}_{m,j}(T^n)$ over $\dot X_{m,j}(T^n)$  given by Equation  \eqref{Eq:mudefinedby}:  
\[\mu(\phi,G_n,x_1):=(\phi^{\mu x_1},\mu G_n,\mu x_1)=(\phi,G_n,\mu x_1).\]
 
Let $x_1,\ldots,x_n$ be the coordinates of $\ddot X_{m,j}(T^n)$ that are subject to  Equation \eqref{Eq:A,m,j}. Then $\mu$  sends 	$(x_1,\ldots,x_n)$ to $(\mu x_1,\mu x_2,\ldots,\mu x_n)$, because   the point $(\mu x_1,\mu x_2,\ldots,\mu x_n)$ satisfies Equation \eqref{Eq:A,m,j}. 
Relying on this fact, we now  prove part B of the Main Theorem. 
\begin{proof}[Proof of part B]
	We prove the claim by induction on $n$.~~
	
	\begin{enumerate}
	\item Case $n=1$.
	
It follows from part (2) of Lemma \ref{Lem:RelDDD}  that  the function field  $\dot{F}^{(1)}_{m.j}$ of 
$\dot{X}_{m,j}(T)$ is generated by $X_1=x_1^{q-1}$.
\item Case $n\geqslant 2$.
Suppose that the function field $\dot{F}^{(n-1)}_{m.j}$ of 
$\dot{X}_{m,j}(T^{n-1})$ is generated by $X_1$, $X_2$, $\ldots$, $X_{n-1}$ with 
$ 
X_i=x_i^{q-1}
$ (for $i=1,\ldots, n-1$)
satisfying Equation (\ref{Eq:dA,m,j}) (up to $(n-1)$). 
Set $X_n=x_n^{q-1}$. Then $X_n$ is fixed by the action of the  multiplicative group $\F^*_q$, and therefore $F^{n-1}_{m,j}(X_n)\subseteq F^n_{m,j}$. In the meantime, it follows from   Equation (\ref{Eq:A,m,j}) for $i=n$ that

\[{\left(\mathcal{F}(x_{n-1},x_n)+1\right)}^{q-1}=1,\]
%	\[\left(\frac{x_n}{x_{n-1}}\right)^{q-1}\left(	\frac{1}{x_{n-1}^{q^k-1}}+\frac{x_n^{q-1}}{x_{n-1}^{q^{k+1}-1}}+\cdots+\frac{x_n^{q^{j-1}-1}}{x_{n-1}^{q^{k+j-1}-1}}+x_n^{q^j-1}+
%	\frac{x_n^{q^{j+1}-1}}{x_{n-1}^{q-1}}+\cdots+\frac{x_n^{q^{m-1}-1}}{x_{n-1}^{q^{k-1}-1}}\right)^{q-1}=1,\]
		and then
	\begin{align*} 
	\left(\frac{1}{X_{n-1}^{N_k}}+\frac{X_n}{X_{n-1}^{N_{k+1}}}+\cdots+\frac{X_n^{N_{j-1}}}{X_{n-1}^{N_{k+j-1}}}+X_n^{N_j}+
					\frac{X_n^{N_{j+1}}}{X_{n-1}}+\cdots+\frac{X_n^{N_{m-1}}}{X_{n-1}^{N_{k-1}}}\right)^{q-1}=\frac{X_{n-1}}{X_n},
					\end{align*}
					i.e.,
					\begin{equation*}
					\mathcal{G}(X_{n-1},X_n)=0.
					\end{equation*}
			Therefore,  $\dot F_{m,j}^{n-1}(x_n)$ is a degree $q^{m-1}$ extension over $\dot F_{m,j}^{n-1}$ . Combining this fact with Corollary \ref{Cor:DXTn},  the function field of $\dot{X}_{m,j}(T^n)$ over $\dot{X}_{m,j}(T^{n-1})$ is generated by $X_1$, $\ldots$, $X_n$ satisfying Equation \eqref{Eq:dA,m,j}. This completes the proof of part B.
					\end{enumerate}
				
 \end{proof}
\subsection{Part C}
We wish to find  the function field of $X_{m,j}(T^n)$ over $\F_{q^m}$.  %Our approach is based on results in  \cite{A.Bassa2014} and \cite{Elkies1997}. A generalized fact (OF WHAT???) is needed.
%\begin{lem}???? \cite{????}
%\end{lem}
%For $\mu\in \F_{q^m}^*$, defining the action of $\mu$ on the point $(x_1,\ldots,x_n)$ as follows
%\begin{equation}
%\mu(x_1,\ldots,x_n):=(\mu x_1,\ldots,\mu^{q^{k(n-1)}}x_n).
%\end{equation}
%By direct calculation, we get that $(\mu x_1,\ldots, \mu^{q^{k(n-1)}}x_n)$ satisfies Equation $(\ref{Eq:A,m,j})$.
%\begin{lem}\label{Lem:GalXX}
%The curve $\ddot{X}_{m,j}(T^n)$ is a Galois covering over  $\ X_{m,j}(T^n)$  and the correspondence \textit{Galois} group  $G$ is isomorphic to
%	$\F_{q^m}^*$.
%\end{lem}
%\begin{proof}
%	Denote by $x_1,\ldots, x_n$ the coordinate of $\ddot{X}_{m,j}(T^n)$ satisfying Equation (\ref{Eq:A,m,j}). For $\mu\in \F_{q^m}^*$, the action of $\mu$  induces an isomorphism from $\ddot{X}_{m,j}(T^n)$ to itself over $X_{m,j}(T^n)$, we also denote by $\mu$. Explicitly
%	\[\mu: \ddot{X}_{m,j}(T^n)\to \ddot{X}_{m,j}(T^n),\]
%	\[(\phi,G_n,x_1)\mapsto (\phi^{\mu x_1},\mu G_n,\mu x_1).\]
%	Combining with Lemma \ref{LL51}, we know that the  Galois  group  $G$ of $\ddot{X}_{m,j}(T^n)$ over $\ X_{m,j}(T^n)$ is isomorphic to
%	$\F_{q^m}^*$.
%\end{proof}
First we recall   some results in \cite{A.Bassa2014} and \cite{Bassa2015}. Let $x$ and $y$ be in $\bar L$,  where $x\neq 0$, and  suppose that they are related by  the equation
\begin{equation}\label{Eq:generalConfine}
\mathcal{F}(x,y)=0.
\end{equation}
 We adopt the following bivariant fractional functions:
\[R(x,y):=\frac{y}{x^{q^k}},\]\[  S(x,y):=\frac{y^{q^j}}{x}, \]
and $$u(x,y):=\sum_{r=0}^{a-1}R^{q^{rk}}+\left(\sum_{s=0}^{b-1}{S}^{q^{sj}}\right)^q.$$ 
(Recall that $a$ and $b$ are two non-negative integers satisfying
$ak-bj=1$. 
)

 %Recall that the non-negative integers $a$ and $b$ satisfy $ak-bj=1$.	
\begin{lem}[{\cite[Proposition $3$]{A.Bassa2014} and  \cite[Proposition $2.2$,~$2.3$]{Bassa2015}}] \label{Lem:RSu}Let $R$, $S$, and $u$ be as above. We have
	the following facts.
\begin{enumerate}
\item   The functions $R$, $S$, and $u$ are related by
\[R=\tr_k(u)-b, \quad S=-\tr_j(u)+a,\]
and, therefore, the function field $\F_{q^m}(R,S)$ is identically $\F_{q^m}(u)$ and rational; % IMPLIED OR IFF???
\item  The function fields  $\F_{q^m}(x,y)$, $\F_{q^m}(u,x)$, and $\F_{q^m}(u,y)$ are one and the same;  
\item   The   field extension  $\F_{q^m}(x,y)/{\F_{q^m}(u)}$ is   Galois, and  its Galois group is isomorphic to $\F_{q^m}^*$.
\end{enumerate}
\end{lem}

Now we are ready to give  the  
\begin{proof}[Proof of part C]     

We  observe a basic fact. 
Let $x_1,\ldots,x_n$ be the coordinates of $\ddot X_{m,j}(T^n)$ that are subject to Equation \eqref{Eq:A,m,j}. Then for each $\mu\in\F_{q^m}^*$, 	
the automorphism $\mu$ defined by \eqref{Eq:mudefinedby}
sends  $(x_1,\ldots,x_n)$ to $(\mu x_1,\mu^{q^k}x_2,\ldots,\mu^{q^{k(n-1)}}x_n)$. Indeed, one can easily check that the data $(\mu x_1,\mu^{q^k}x_2,\ldots,\mu^{q^{k(n-1)}}x_n)$ is subject to Equation \eqref{Eq:A,m,j}.  
	
	We proceed to accomplish the proof. 
	\begin{enumerate}  
				\item Case $n=1$.

		Since $\ddot{X}_{m,j}(T)$ is a Galois covering over $X_{m,j}(T)$ of degree $(q^{m}-1)$,  the function field  $F^{(1)}_{m,j}$ of $X_{m,j}(T)$ is generated by $z=x_1^{q^m-1}$.
		\item Case $n=2$. In this situation, $(x_1,x_2)=(x,y)$ is subject to Equation \eqref{Eq:generalConfine}. We adopt the notations $R_2=R(x_1,x_2)$,  $S_2=S(x_1,x_2)$, and
		$u_2=u(x_1,x_2)$. 
By direct calculations, we see that $R_2$ and $S_2$ are stable under the action by $\mu$. It follows from Lemmas \ref{Lem:RelDDD} (part (3)) and \ref{Lem:RSu} (parts (1) and (3)) that the function field $F^{(2)}_{m,j}$ of $X_{m,j}(T^2)$ is equal to $\F_{q^m}(R_2,S_2)=\F_{q^m}(u_2)$.
		
		\item Case $n\geqslant 3$.  We induct on $n$ from the preceding $n=2$ case. Assume that in the $(n-1)$-step,  the function field $F^{(n-1)}_{m,j}$ of $X_{m,j}(T^{n-1})$ over $\F_{q^m}$, is generated by variables $u_2,\ldots,u_{n-1}$, with $u_i=u(x_{i-1},x_i)$.

		Consider variables 
		\[R_n:=R(x_{n-1},x_n), \text{  } S_n:=S(x_{n-1},x_n),  \mbox{ and } u_n:=u(x_{n-1},x_n).\] 
	
		By part (2) of Lemma \ref{Lem:RSu}, we have   $\ddot{F}_{m,j}^{(n)}=F^{(n-1)}_{m,j}(x_{n-1},x_n)$. Then along the same lines as in the  $n=2$ case, we  use  Lemmas \ref{Lem:RelDDD} and \ref{Lem:RSu},  and see that  the function field $F^{(n)}_{m,j}$ of $X_{m,j}(T^n)$ must be identical to  $F^{(n-1)}_{m,j}(u_n)$.
		
	According to the definitions of $S$ and $R$, we have
		\begin{equation*}%\label{2}
		x_{n-1}^{q^m-1}=\frac{S_{n-1}^{q^k}}{R_{n-1}}=\frac{S_n}{R_n^{q^j}}.
		\end{equation*}
		Finally, we take advantage of   part (1) of 	Lemma \ref{Lem:RSu}, and get a relation between $u_{n-1}$ and $u_n$:
		\begin{equation*}%\label{Eq:u}
		\frac{\tr_j(u_{n-1})^{q^k}-a}{\tr_k(u_{n-1})-b}=\frac{\tr_j(u_n)-a}{\tr_k(u_n)^{q^j}-b}.
		\end{equation*}
		%It implies that $F^{(n)}_{m,j}=F^{(n-1)}_{m,j}(u_n)$.
		This proves Equation \eqref{Eq:HU1U2}.   The induction proceeds and  the proof is complete.
	\end{enumerate}
\end{proof}
%\begin{remark}
\begin{remark}\label{Rmk:last}

We conclude this paper by an easy application --- to interpret the lower bound of  $\F_{q^m}$-rational   points appeared in the  BBGS towers.

	 Let $x_1,\ldots,x_n$ be coordinates on $\ddot{X}_{m,j}(T^n)$ that are subject to Equation \eqref{Eq:A,m,j}.
From the standpoint of Drinfeld modular curves,  the condition   ``$x_1$ belongs to $\F_{q^m}^*$'' is equivalent to   ``$(\phi,G_n,x_1)$ is a supersingular point,'' (i.e., $\phi$ is supersingular).  The covering map $\ddot{X}_{m,j}(T^n) \to \ddot{X}_{m,j}(T)$ is exactly the projection $(x_1,\ldots, x_n)$ to $x_1$. Thus there are as many as $(q^m-1)q^{(m-1)(n-1)}$ points on $\ddot{X}_{m,j}(T^n)$ which are supersingular. This recovers the result obtained by Bassa \textit{et al.} \cite[Corollary 3.2]{Bassa2015}.

For $n\geqslant 2$, let $u_2,\ldots,u_n$ be the coordinates of $X_{m,j}(T^n)$ that are subject to Equation \eqref{Eq:dA,m,j}. From our description of    supersingular points on $\ddot X_{m,j}(T^n)$, we see that these $q^{(m-1)(n-1)}$ supersingular points on  $X_{m,j}(T^n)$ comprise the set   
\[\left\lbrace  (u_2,\ldots,u_n)\in(\F_{q^m}^*)^{n-1}|~\tr_m(u_i)=a+b, \quad i=2,\ldots,n\right\rbrace.\]

\end{remark}

%\end{remark}

%\section{to recycle}
%
%By Lemma \ref{Lem:Isomorphic}, every $L$-point of the  curve $X_{m,j}(N)$ can be represented by a pair $(\phi,G)$, where $\phi$ is a normalized Drinfeld module.  Recall that in Definition \ref{Defn:curve}, two pairs  $ (\phi, G) $ and $ (\phi', G') $    are said to be equivalent, if  $ g_j = g_j'  \lambda^{q^j - 1} $ and $  \lambda G =  G' $  for some $ \lambda \in  \mathbb{F}_{q^m}^{* } $.  Here the notations are as in \eqref{Eq:phiTandphiT'}.

\bibliographystyle{90}
% argument is your BibTeX tring definitions and bibliography database(s)
\bibliography{papers}

\begin{bibdiv}
\begin{biblist}

\bib{Aleshnikov1999}{article}{
      author={Aleshnikov, Ilia},
      author={Deolalikar, Vinay},
      author={Kumar, P.~Vijay},
      author={Stichtenoth, Henning},
       title={Towards a basis for the space of regular functions in a tower of
  function fields meeting the {D}rinfeld-{V}ladut bound},
        date={2001},
       pages={14\ndash 24},
      review={\MR{1849075}},
}

\bib{Nurdagul2017}{article}{
      author={Anbar, Nurdag\"{u}l},
      author={Bassa, Alp},
      author={Beelen, Peter},
       title={A modular interpretation of various cubic towers},
        date={2017},
        ISSN={0022-314X},
     journal={J. Number Theory},
      volume={171},
       pages={341\ndash 357},
         url={https://doi.org/10.1016/j.jnt.2016.07.025},
      review={\MR{3556689}},
}

\bib{Angles2002}{article}{
      author={Angles, Bruno},
      author={Maire, Christian},
       title={A note on tamely ramified towers of global function fields},
        date={2002},
     journal={Finite Fields Appl.},
      volume={8},
      number={2},
       pages={207–215},
}

\bib{A.Bassa2014}{article}{
      author={Bassa, Alp},
      author={Beelen, Peter},
      author={Garcia, Arnaldo},
      author={Stichtenoth, Henning},
       title={{Galois} towers over non-prime finite fields},
        date={2014},
        ISSN={1609-3321},
     journal={Acta Arith.},
      volume={164},
      number={2},
       pages={163–179},
      review={\MR{3224833}},
}

\bib{Bassa2015}{article}{
      author={Bassa, Alp},
      author={Beelen, Peter},
      author={Garcia, Arnaldo},
      author={Stichtenoth, Henning},
       title={Towers of function fields over non-prime finite fields},
        date={2015},
        ISSN={1609-3321},
     journal={Mosc. Math. J.},
      volume={15},
      number={1},
       pages={1\ndash 29, 181},
      review={\MR{3427409}},
}

\bib{Beelen2018}{article}{
      author={Beelen, Peter},
      author={Montanucci, Maria},
       title={A new family of maximal curves},
        date={2018},
     journal={J. Lond. Math. Soc.},
      volume={2},
      number={3},
       pages={573\ndash 592},
      review={\MR{3893192}},
}

\bib{Bezerra2005}{article}{
      author={Bezerra, Juscelino},
      author={Garcia, Arnaldo},
      author={Stichtenoth, Henning},
       title={An explicit tower of function fields over cubic finite fields and
  {Z}ink’s lower bound},
        date={2005},
     journal={J. ReineAngew. Math.},
      volume={589},
       pages={159\ndash 199},
}

\bib{Dale2018}{article}{
      author={Brownawell, W.~Dale},
      author={Papanikolas, Matthew~A.},
       title={A rapid introduction to {Drinfeld} modules, t-modules, and
  t-motives},
        date={2018},
     journal={arXiv:1806.03919 [math.NT]},
}

\bib{Cascuso2014}{article}{
      author={Cascudo, Ignacio},
      author={Cramer, Ronald},
      author={Xing, Chaoping},
       title={Torsion limits and {R}iemann-{R}och systems for function fields
  and applications},
        date={2014},
     journal={IEEE Trans. Inform. Theory},
      volume={60},
       pages={3871\ndash 3888},
      review={\MR{3225937}},
}

\bib{Cakcak2004}{article}{
      author={\c{C}ak\c{c}ak, Emrah},
      author={\"{O}zbudak, Ferruh},
       title={Subfields of the function field of the {Deligne–Lusztig} curve
  of {Ree} type},
        date={2004},
     journal={Acta Arith.},
      volume={115},
      number={2},
       pages={133–180},
      review={\MR{2099835}},
}

\bib{Cakcak2005}{article}{
      author={\c{C}ak\c{c}ak, Emrah},
      author={\"{O}zbudak, Ferruh},
       title={Number of rational places of subfields of the function field of
  the {Deligne–Lusztig} curve of {Ree} type},
        date={2005},
     journal={Acta Arith.},
      volume={120},
      number={1},
       pages={79–106},
      review={\MR{2189720}},
}

\bib{Drinfeld1974}{article}{
      author={Drinfeld, Vladimir~Gershonovich},
       title={Elliptic modules},
        date={1974},
     journal={Mat. Sb. (N.S.)},
      volume={94(136)},
       pages={594–627, 656},
      review={\MR{0384707}},
}

\bib{Elkies1997}{article}{
      author={Elkies, Noam~D.},
       title={Explicit modular towers},
        date={1997},
     journal={in Proc. 35th Ann. Allerton Conf. on Communication, Control and
  Computing, Urbana, IL},
       pages={23–32},
}

\bib{Elkies2001}{article}{
      author={Elkies, Noam~D.},
       title={Explicit towers of {D}rinfeld modular curves},
        date={2001},
      volume={202},
       pages={189\ndash 198},
      review={\MR{1905359}},
}

\bib{Elkies2002}{book}{
      author={Elkies, Noam~D.},
       title={Appendix to new optimal tame towers of function fields over small
  finite fields by {W.-C. W. Li}, {H. M}aharaj, and {H. S}tichtenoth},
      series={Springer-Verlag, Berlin},
   publisher={Lecture Notes in Computer Science 2369, C. Fieker and D. R.
  Kohel, eds},
        date={2002},
}

\bib{Garcia2010}{article}{
      author={Garcia, Arnaldo},
      author={G\"{u}neri, Cem},
      author={Stichtenoth, Henning},
       title={A generalization of the {G}iulietti-{K}orchm\'{a}ros maximal
  curve},
        date={2010},
     journal={Advances in Geometry},
      volume={10},
      number={3},
       pages={427–434},
      review={\MR{2660419}},
}

\bib{Garcia1995}{article}{
      author={Garcia, Arnaldo},
      author={Stichtenoth, Henning},
       title={A tower of {A}rtin-{S}chreier extensions of function fields
  attaining the {Drinfeld-Vl{ă}du{ţ}} bound},
        date={1995},
     journal={Invent. Math.},
      volume={121},
       pages={211\ndash 222},
}

\bib{GarciaStichtenoth1996}{article}{
      author={Garcia, Arnaldo},
      author={Stichtenoth, Henning},
       title={On the asymptotic behaviour of some towers of function fields
  over finite fields},
        date={1996},
     journal={Journal of Number Theory},
      volume={61},
      number={2},
       pages={248\ndash 273},
      review={\MR{1423052}},
}

\bib{Garcia2003}{article}{
      author={Garcia, Arnaldo},
      author={Stichtenoth, Henning},
      author={R\"{u}ck, Hans-Georg},
       title={On tame towers over finite fields},
        date={2003},
     journal={J. Reine Angew. Math},
      volume={557},
       pages={53\ndash 80},
      review={\MR{1978402}},
}

\bib{Gekeler1986}{book}{
      author={Gekeler, Ernst-Ulrich},
       title={{Drinfeld} modular curves},
      series={Lecture Notes in Mathematics, 1231. Springer-Verlag, Berlin},
        date={1986},
}

\bib{Gekeler2004}{article}{
      author={Gekeler, Ernst-Ulrich},
       title={Asymptotically optimal towers of curves over finite fields},
        date={2004},
     journal={In Algebra, Arithmetic and Geometry with Applications,
  Springer-Verlag, Berlin Heidelberg,},
       pages={325–336},
      review={\MR{2037099}},
}

\bib{Gekeler2019}{article}{
      author={Gekeler, Ernst-Ulrich},
       title={Towers of {$GL(r)$}-type of modular curves},
        date={2019},
     journal={J. Reine Angew. Math.},
      volume={350},
      review={\MR{4000571}},
}

\bib{Giulietti2008}{article}{
      author={Giulietti, Massimo},
      author={Korchm\'{a}ros, G\'{a}bor},
       title={A new family of maximal curves over a finite field},
        date={2008},
     journal={Mathematische Annalen},
      volume={343},
      number={1},
       pages={229–245},
      review={\MR{2448446}},
}

\bib{Giulietti2006}{article}{
      author={Giulietti, Massimo},
      author={Korchm\'{A}ros, G\'{a}bor},
      author={Torres, Fernando},
       title={Quotient curves of the {Suzuki} curve},
        date={2006},
     journal={Acta Arith.},
      volume={122},
       pages={245–274},
}

\bib{Goppa1981}{article}{
      author={Goppa, Valery~Denisovich},
       title={Codes on algebraic curves},
        date={1981},
     journal={Soviet Math. Dokl},
      volume={259},
      number={6},
       pages={1289–1290},
      review={\MR{0628795}},
}

\bib{Goss1996}{book}{
      author={Goss, David},
       title={Basic structures of function field arithmetic},
      series={Springer},
   publisher={Berlin Heidelberg},
        date={1996},
      review={\MR{1423131}},
}

\bib{Hajir2000}{article}{
      author={Hajir, Farshid},
      author={Maire, Christian},
       title={Asymptotically good towers of global fields},
        date={2001},
      volume={202},
       pages={207\ndash 218},
      review={\MR{1905361}},
}

\bib{Hall2013}{article}{
      author={Hall-Seelig, Laura~L.},
       title={New lower bounds for the {I}hara function {$A(q)$} for small
  primes},
        date={2013},
        ISSN={0022-314X},
     journal={J. Number Theory},
      volume={133},
      number={10},
       pages={3319\ndash 3324},
      review={\MR{3071814}},
}

\bib{Hallouin2016}{article}{
      author={Hallouin, Emmanuel},
      author={Perret, Marc},
       title={A graph aided strategy to produce good recursive towers over
  finite fields},
        date={2016},
     journal={Finite Fields Appl},
      volume={42},
       pages={200\ndash 224},
      review={\MR{3550391}},
}

\bib{Hasegawa2013}{article}{
      author={Hasegawa, Takehiro},
       title={Remarks on a paper by {Maharaj} and {Wulftange}},
        date={2013},
     journal={J. Pure Appl. Algebra},
      volume={217},
      number={1},
       pages={1\ndash 10},
      review={\MR{2965897}},
}

\bib{Hasegawa2017}{article}{
      author={Hasegawa, Takehiro},
       title={An explicit {Shimura} tower of function fields over a number
  field: An application of {Takeuchi's} list},
        date={January 2017},
     journal={arXiv:1701.07551, math.NT, math.AG},
}

\bib{Hasegawa2012}{article}{
      author={Hasegawa, Takehiro},
      author={Inuzuka, Miyoko},
      author={Suzuki, Takafumi},
       title={Towers of function fields over finite fields corresponding to
  elliptic modular curves},
        date={2012},
     journal={Finite Fields and their Applications},
      volume={18},
      number={1},
       pages={1\ndash 18},
      review={\MR{2874901}},
}

\bib{Hu2017}{article}{
      author={Hu, Chuangqiang},
       title={Explicit construction of {AG} codes from a curve in the tower of
  {B}assa-{B}eelen-{G}arcia-{S}tichtenoth},
        date={2017},
     journal={IEEE Trans. Inform. Theory},
      volume={63},
      number={11},
       pages={7237–7246},
      review={\MR{3724425}},
}

\bib{Hu2016}{article}{
      author={Hu, Chuangqiang},
      author={Zhao, Chang-An},
       title={Multi-point codes from generalized {H}ermitian curves},
        date={2016},
     journal={IEEE Trans. Inform. Theory},
      volume={62},
      number={5},
       pages={2726–2736},
      review={\MR{3493874}},
}

\bib{Ihara1982}{article}{
      author={Ihara, Yasutaka},
       title={Some remarks on the number of rational points of algebraic curves
  over finite fields},
        date={1981},
     journal={In Journal of the Faculty of Science, the University of Tokyo,
  Sect. 1 A, Mathematics},
      volume={28},
       pages={721–724, 1982},
      review={\MR{0656048}},
}

\bib{Li2002}{article}{
      author={Li, Wen-Ching~W.},
      author={Maharaj, Hiren},
       title={Coverings of curves with asymptotically many rational points},
        date={2002},
     journal={J. Number Theory},
      volume={96(2)},
       pages={232–256},
      review={\MR{1932454}},
}

\bib{H.Stichtenoth2002}{book}{
      author={Li, Wen-Ching~W.},
      author={Maharaj, Hiren},
      author={Stichtenoth, Henning},
      author={Elkies, Noam~D.},
       title={New optimal tame towers of function fields over small finite
  fields},
      series={Springer-Verlag, Berlin},
   publisher={Lecture Notes in Computer Science 2369, C.Fieker and %D.R.Kohel,
  eds},
        date={2002},
      review={\MR{2041098}},
}

\bib{Niederreiter1998}{article}{
      author={Niederreiter, Harald},
      author={Xing, Chaoping},
       title={Towers of global function fields with asymptotically many
  rational places and an improvement on the {G}ilbert–{V}arshamov bound},
        date={1998},
     journal={Math. Nachr.},
      volume={195},
       pages={171–186},
      review={\MR{1654693}},
}

\bib{Niederreiter2001}{book}{
      author={Niederreiter, Harald},
      author={Xing, Chaoping},
       title={Rational points of curves over finite fields},
      series={Cambridge University Press},
   publisher={Cambridge},
        date={2001},
}

\bib{Ore1933}{article}{
      author={Ore, Oystein},
       title={Theory of noncommutative polynomials},
        date={1933},
     journal={Annals of Math},
      volume={34},
       pages={480\ndash 508},
}

\bib{Serre1984}{article}{
      author={Serre, Jean-Pierre},
       title={Résumé des cours de 1983-1984},
        date={1984},
     journal={in Annuaire College de France},
      number={128},
       pages={79–83},
}

\bib{Serre1983}{article}{
      author={Serre, Jean-Pierre},
       title={Sur le nombre des points rationnels d’une courbe algébrique
  sur un corps fini},
        date={1984},
     journal={C. R. Acad. Sci. Paris},
      volume={296},
       pages={397–402},
}

\bib{Skabelund2018}{article}{
      author={Skabelund, Dane~C.},
       title={New maximal curves as ray class fields over deligne-lusztig
  curves},
        date={2018},
     journal={Proc. Amer. Math. Soc.},
      volume={146},
      number={2},
       pages={525\ndash 540},
      review={\MR{3731688}},
}

\bib{H.Stichtenoth2009}{book}{
      author={Stichtenoth, Henning},
       title={Algebraic function fields and codes},
     edition={Second},
      series={Graduate Texts in Mathematics},
   publisher={Springer-Verlag, Berlin},
        date={2009},
      volume={254},
        ISBN={978-3-540-76877-7},
      review={\MR{2464941}},
}

\bib{Tsfasman1982}{article}{
      author={Tsfasman, Mikhail~Anatolievich},
      author={Vl{\u{a}}du{\c{t}}, Sergei~G.},
      author={Zink, Thomas},
       title={Modular curves, {S}himura curves, and {G}oppa codes, better than
  {V}arshamov-{G}ilbert bound},
        date={1982},
     journal={Mathematische Nachrichten},
      volume={109(1)},
       pages={21–28},
      review={\MR{0705893}},
}

\bib{Villa2006}{book}{
      author={Villa~Salvador, Gabriel~Daniel},
       title={Topics in the theory of algebraic function fields},
        date={2006},
      review={\MR{2241963}},
}

\bib{Vladut1983}{article}{
      author={Vl{ă}du{ţ}, Sergei~G.},
      author={Drinfeld, Vladimir~Gershonovich},
       title={The number of points of an algebraic curve},
        date={1983},
     journal={Funktsional. Anal. i Prilozhen.},
      volume={17},
      number={1},
       pages={68\ndash 69},
      review={\MR{0695100}},
}

\bib{Conny1997}{article}{
      author={Voss, Conny},
      author={H{\o{}}holdt, Tom},
       title={An explicit construction of a sequence of codes attaining the
  {Tsfasman-Vlăduţ-Zink} bound: the first steps},
        date={1997},
     journal={IEEE Trans. Inform. Theory},
      volume={43},
      number={1},
       pages={128–135},
      review={\MR{1426240}},
}

\bib{Weil1971}{book}{
      author={Weil, Andr\'{e}},
       title={Courbes algébriques et variétés abéliennes},
   publisher={Herman, Paris},
        date={1971},
}

\bib{Xing2007}{article}{
      author={Xing, Chaoping},
      author={Yeo, Sze~Ling},
       title={Algebraic curves with many points over the binary field},
        date={2007},
     journal={J. Algebra},
      volume={311},
       pages={775–780},
      review={\MR{2314733}},
}

\bib{Hu2019}{article}{
      author={Yang, Shudi},
      author={Hu, Chuangqiang},
       title={Weierstrass semigroups from a tower of function fields attaining
  the {Drinfeld-Vl{ă}du{ţ}} bound},
        date={2019},
     journal={arXiv:1911.04269 [math.NT]},
}

\bib{Zink1985}{book}{
      author={Zink, Thomas},
       title={Degeneration of {S}himura surfaces and a problem in coding
  theory},
      series={Fundamentals of computation theory (Cottbus, 1985), 503–511,
  Lecture Notes in Comput. Sci., 199, Springer, Berlin},
        date={1985},
      review={\MR{0821267}},
}

\end{biblist}
\end{bibdiv}
\end{document}